\documentclass[leqno,12pt]{article} 
\usepackage{amsmath, amssymb}
\usepackage{amsthm} 
\usepackage{pictexwd,dcpic}
\usepackage{xfrac}
\usepackage{xypic}

\theoremstyle{plain} 
\newtheorem{theorem}{\indent\sc Theorem}[section]
\newtheorem{lemma}[theorem]{\indent\sc Lemma}
\newtheorem{corollary}[theorem]{\indent\sc Corollary}
\newtheorem{proposition}[theorem]{\indent\sc Proposition}

\theoremstyle{definition} 

\newtheorem{remark}[theorem]{\indent\sc Remark}
\newtheorem{example}[theorem]{\indent\sc Example}
\newtheorem{notation}[theorem]{\indent\sc Notation}

\numberwithin{equation}{section}

\def\F{\mathfrak{f}}
\def\P{\mathfrak{p}}

\def\o{\omega}

\def\G{\Gamma}
\def\O{\mathcal{O}}

\makeatletter
\def\address#1#2{\begingroup
\noindent\parbox[t]{7.8cm}{%
\small{\scshape\ignorespaces#1}\par\vskip1ex
\noindent\small{\itshape E-mail address}%
\/: #2\par\vskip4ex}\hfill%
\endgroup}%

\makeatother
\title{\uppercase{Construction of ray class fields by smaller generators and applications}} 
\author{
\textsc{Ja Kyung Koo and Dong Sung Yoon} 
}

\date{} 
\begin{document}

\maketitle

 \footnote{
\thanks{
The first named author was supported by the NRF of Korea grant funded by the Korea government (MISP) (No.2014001824).
The second named author was supported by the National Institute for Mathematical Sciences, Republic of Korea.
} }

\begin{abstract}
We first generate ray class fields over imaginary quadratic fields in terms of Siegel-Ramachandra invariants, which would be an extension of Schertz's result \cite{Schertz}.
And, by making use of quotients of Siegel-Ramachandra invariants we also construct ray class invariants over imaginary quadratic fields whose minimal polynomials have relatively small coefficients, from which we are able to solve certain quadratic Diophantine equations.

\end{abstract}

\tableofcontents

\maketitle

\section{Introduction}

Let $K$ be an imaginary quadratic field, $\mathfrak{f}$ be a nonzero integral ideal of $K$ and 
$\mathrm{Cl}(\F)$ be the ray class group of $K$ modulo $\mathfrak{f}$.
Then by class field theory there exists a unique abelian extension of $K$ whose Galois group is isomorphic to $\mathrm{Cl}(\F)$ via the Artin map
\begin{equation}\label{artin map}
\sigma:\mathrm{Cl}(\F)\xrightarrow{~\sim~} \mathrm{Gal}(K_{\F}/K).
\end{equation}
We call it the \textit{ray class field} of $K$ modulo $\mathfrak{f}$ which is denoted by $K_\mathfrak{f}$.
Since any abelian extension of $K$ is contained in some ray class field $K_\mathfrak{f}$, 
generation of ray class fields of $K$ is the key step toward the Hilbert's 12th problem.
In 1964, Ramachandra (\cite[Theorem 10]{Ramachandra}) constructed a primitive generator of $K_\F$ over $K$ by applying the Kronecker limit formula.
However, his invariants involve products of high powers of singular values of the Klein forms and the discriminant $\Delta$-function, which are quite complicated to use in practice.
On the other hand, Schertz tried to find rather simpler generators of $K_\F$ over $K$ for practical use.
And, he conjectured that the Siegel-Ramachandra invariants would be the right answer and gave a conditional proof (\cite[Theorem 3 and 4]{Schertz}).
\par
In this paper we shall first generate ray class fields $K_\F$ over $K$ via Siegel-Ramachandra invariants by improving Schertz's idea (Theorem \ref{main theorem}).
And, by making use of quotient of Siegel-Ramachandra invariants we shall also construct a primitive generator of $K_\F$ over $K$ whose minimal polynomial has relatively small coefficients and present several examples 
(Theorem \ref{main theorem2}, Remark \ref{additional case} and Example \ref{first example}, \ref{even example}).
This ray class invariant becomes a real algebraic integer with certain conditions (Lemma \ref{real generator} and Theorem \ref{real ray class invariant}).
Lastly,  we will apply the real ray class invariant to solving certain quadratic Diophantine equations 
(Theorem \ref{real ray class invariant} and Example \ref{last example}, \ref{last example2}).

\begin{notation}
For $z\in\mathbb{C}$ we denote by $\overline{z}$ the complex conjugate of $z$ and by Im$(z)$ the imaginary part of $z$, and put $q_z=e^{2\pi iz}$.
If $G$ is a group and $g_1,g_2,\ldots,g_r$ are elements of $G$, let $\langle g_1,g_2,\ldots,g_r \rangle$ be the subgroup of $G$ generated by $g_1,g_2,\ldots,g_r$, and $G^n$ be the subgroup $\{g^n~|~g\in G \}$ of G for $n\in\mathbb{Z}_{>0}$.
Moreover, if $H$ is a subgroup of $G$ and $g\in G$, we mean by $[g]$ the coset $gH$ of $H$ in $G$.
The transpose of a matrix $\alpha$ is denoted by ${^t}\alpha$.
If $R$ is a ring with identity, $R^\times$ indicates the group of all invertible elements of $R$.
For a number field $K$, let $\O_K$ be the ring of algebraic integers of $K$ and $d_K$ be the discriminant of $K$.
If $a\in\O_K$, we denote by $(a)$ the principal ideal of $K$ generated by $a$.
When $\mathfrak{a}$ is an integral ideal of $K$, we mean by $\mathcal{N}(\mathfrak{a})$ the absolute norm of an ideal $\mathfrak{a}$.
For a positive integer $N$, we let $\zeta_N=e^{2\pi i/N}$ be a primitive $N$-th root of unity.
\end{notation}

\section{Shimura's reciprocity law}

We shall review an algorithm for finding all conjugates of the special value of a modular function over an imaginary quadratic field by utilizing Shimura's reciprocity law.
\par
For a lattice $L$ in $\mathbb{C}$, the \textit{Weierstrass $\P$-function} is defined by
\begin{equation*}
\P(z;L)=\frac{1}{z^2}+\sum_{\o\in L\setminus\{0\}}\left(\frac{1}{(z-w)^2}-\frac{1}{w^2} \right)\quad(z\in\mathbb{C}) .
\end{equation*}
Let 
$\mathbb{H}=\{z\in \mathbb{C}~|~\mathrm{Im}(z)>0\}$ be the complex upper half-plane.
For $\tau\in\mathbb{H}$, we let
\begin{eqnarray*}
g_2(\tau)=60\sum_{\o\in[\tau,1]\setminus\{0\}}\frac{1}{w^4},\quad g_3(\tau)=140\sum_{\o\in[\tau,1]\setminus\{0\}}\frac{1}{w^6},\quad
\Delta(\tau)=g_2(\tau)^3-27g_3(\tau)^2.
\end{eqnarray*}
Then the \textit{$j$-invariant} is defined by
\begin{equation*}
j(\tau)=1728\frac{g_2(\tau)}{\Delta(\tau)}\quad (\tau\in\mathbb{H}).
\end{equation*}
For a rational vector $\left[\begin{matrix}r_1\\r_2\end{matrix}\right]\in\mathbb{Q}^2\setminus\mathbb{Z}^2$, we define the \textit{Fricke function} by
\begin{equation*}
f_{\left[\begin{smallmatrix}r_1\\r_2\end{smallmatrix}\right]}(\tau)=-2^7 3^5\frac{g_2(\tau)g_3(\tau)}{\Delta(\tau)}\P(r_1\tau+r_2;[\tau,1])\qquad(\tau\in\mathbb{H}).
\end{equation*}
And, for a positive integer $N$, let
\begin{eqnarray*}
\G(N)&=&\left\{\left[
\begin{matrix}
a&b\\
c&d
\end{matrix}\right] \in SL_{2}(\mathbb{Z})~\big|~
\left[\begin{matrix}
a&b\\
c&d
\end{matrix}\right]
\equiv 
\left[\begin{matrix}
1&0\\
0&1
\end{matrix}\right]
(\bmod{N}) \right\}\\
\mathcal{F}_N&=&\mathbb{Q}\left(j(\tau),f_{\left[\begin{smallmatrix}r_1\\r_2\end{smallmatrix}\right]}(\tau) : \left[\begin{matrix}r_1\\r_2\end{matrix}\right]\in\frac{1}{N}\mathbb{Z}^2\setminus\mathbb{Z}^2\right).
\end{eqnarray*}
We call $\mathcal{F}_N$ the \textit{modular function field of level $N$} over $\mathbb{Q}$.
Then the function field $\mathbb{C}(X(N))$ on the modular curve $X(N)=\G(N)\backslash(\mathbb{H}\cup \mathbb{P}^1(\mathbb{Q}))$ is equal to $\mathbb{C}\mathcal{F}_N$, and $\mathcal{F}_N$ consists of all functions in $\mathbb{C}(X(N))$ whose Fourier coefficients lie in the cyclotomic field $\mathbb{Q}(\zeta_N)$ (\cite[Chapter 6 \S3]{Lang}).
As is well-known, $\mathcal{F}_N$ is a Galois extension of $\mathcal{F}_1=\mathbb{Q}(j(\tau))$ and 
\begin{equation}\label{decomposition}
\mathrm{Gal}(\mathcal{F}_N/\mathcal{F}_1)\cong \mathrm{GL}_2(\mathbb{Z}/N\mathbb{Z})/\{\pm I_2\}\cong G_N\cdot \mathrm{SL}_2(\mathbb{Z}/N\mathbb{Z})/\{\pm I_2\},
\end{equation}
where 
\begin{equation*}
G_N=\left\{\left[
\begin{matrix}
1&0\\
0&d
\end{matrix}\right]~|~d\in(\mathbb{Z}/N\mathbb{Z})^\times
\right\} .
\end{equation*}
More precisely, the element 
$\left[
\begin{matrix}
1&0\\
0&d
\end{matrix}\right]\in G_N$ acts on $\mathcal{F}_N$ by
\begin{equation}\label{diagonal action}
\sum_{n\gg-\infty}c_n q_\tau^{\frac{n}{N}}\longmapsto \sum_{n\gg-\infty}c_n^{\sigma_d} q_\tau^{\frac{n}{N}},
\end{equation}
where $\sum_{n\gg-\infty}c_n q_\tau^{\frac{n}{N}}$ is the Fourier expansion of a function in $\mathcal{F}_N$ and 
$\sigma_d\in\mathrm{Gal}(\mathbb{Q}(\zeta_N)/\mathbb{Q})$ satisfies $\zeta_N^{\sigma_d}=\zeta_N^d$.
And, $\gamma\in \mathrm{SL}_2(\mathbb{Z}/N\mathbb{Z})/\{\pm I_2\}$ acts on $h\in\mathcal{F}_N$ by 
\begin{equation*}
h^\gamma(\tau)=h(\widetilde{\gamma}\tau), 
\end{equation*}
where $\widetilde{\gamma}$ is a preimage of $\gamma$ of the reduction $\mathrm{SL}_2(\mathbb{Z})\rightarrow \mathrm{SL}_2(\mathbb{Z}/N\mathbb{Z})/\{\pm I_2\}$
(\cite[Chapter 6, Theorem 3]{Lang}).
\par

Now let $K$ be an imaginary quadratic field of discriminant $d_K$ and set
\begin{equation}\label{theta}
\theta=\left\{
\begin{array}{ll}
\displaystyle\frac{\sqrt{d_K}}{2}&\textrm{if $d_K\equiv 0\pmod{4}$}\vspace{0.2cm}\\
\displaystyle\frac{-1+\sqrt{d_K}}{2}&\textrm{if $d_K\equiv 1\pmod{4}$}
\end{array}\right.
\end{equation}
so that $\O_K=\mathbb{Z}[\theta]$.
Then its minimal polynomial over $\mathbb{Q}$ is
\begin{equation*}
\textrm{min}(\theta,\mathbb{Q})=X^2+B_\theta X+C_\theta=\left\{
\begin{array}{ll}
X^2-d_K/4&\textrm{if $d_K\equiv 0\pmod{4}$}\vspace{0.2cm}\\
X^2+X+(1-d_K)/4&\textrm{if $d_K\equiv 1\pmod{4}$}.
\end{array}
\right.
\end{equation*}

\begin{proposition}\label{complex multiplication}
When $\F=N\O_K$ for a positive integer $N$, we have
\begin{equation*}
K_{\F}=K_{(N)}=K(h(\theta):~h\in\mathcal{F}_{N,\theta})
\end{equation*}
where $\mathcal{F}_{N,\theta}=\{h\in\mathcal{F}_N~|~ \textrm{$h$ is defined and finite at $\theta$}\} $.
If $N=1$, then $K_{(1)}$ is nothing but the Hilbert class field of $K$.
\end{proposition}
\begin{proof}
\cite[Chapter 10 \S1, Corollary]{Lang}.
\end{proof}

For a positive integer $N$ we define a subgroup $W_{N,\theta}$ of $\mathrm{GL}_2(\mathbb{Z}/N\mathbb{Z})$ by
\begin{equation*}
W_{N,\theta}=\left\{\left[
\begin{matrix}
t-B_\theta s& -C_\theta s\\
s& t
\end{matrix}\right]
\in\mathrm{GL}_2(\mathbb{Z}/N\mathbb{Z})~\big|~t,s\in\mathbb{Z}/N\mathbb{Z}   \right\} .
\end{equation*}
Then we have the following proposition.

\begin{proposition}\label{ray class conjugate}
We attain a surjective homomorphism
\begin{eqnarray*}
\varphi_{N,\theta}: W_{N,\theta}&\longrightarrow& \mathrm{Gal}(K_{(N)}/K_{(1)})\\
\alpha&\longmapsto& (h(\theta)\mapsto h^\alpha(\theta))_{h\in\mathcal{F}_{N,\theta}},
\end{eqnarray*}
whose kernel is
\begin{equation*}
\begin{array}{ll}
\left\{\pm\left[
\begin{matrix} 
1&0\\
0&1
\end{matrix}\right]
\right\}&\textrm{if $K\neq \mathbb{Q}(\sqrt{-1}), \mathbb{Q}(\sqrt{-3})$},\vspace{0.2cm}\\
\left\{\pm\left[
\begin{matrix} 
1&0\\
0&1
\end{matrix}\right],
\pm\left[\begin{matrix} 
0&-1\\
1&0
\end{matrix}\right]
\right\}&\textrm{if $K= \mathbb{Q}(\sqrt{-1})$},\vspace{0.2cm}\\
\left\{\pm\left[\begin{matrix} 
1&0\\
0&1
\end{matrix}\right],
\pm\left[\begin{matrix} 
-1&-1\\
1&0
\end{matrix}\right],
\pm\left[\begin{matrix} 
0&1\\
-1&-1
\end{matrix}\right]
\right\}&\textrm{if $K=\mathbb{Q}(\sqrt{-3})$}.\vspace{0.2cm}
\end{array}
\end{equation*}
\end{proposition}

\begin{proof}
\cite[\S3]{Gee}, \cite[\S3]{Stevenhagen} or \cite[Proposition 2.3]{Jung}.
\end{proof}

On the other hand, observe that the subgroup
\begin{equation*}
\big\{[(\o)]\in \mathrm{Cl}(N\O_K)~|~ \textrm{$\o\in\O_K$ is prime to $N$} \big\}
\end{equation*}
is isomorphic to $\mathrm{Gal}(K_{(N)}/K_{(1)})$ via the Artin map $\sigma$ in (\ref{artin map}).

\begin{proposition}\label{principal class action}
Let $N$ be a positive integer and $\omega\in\O_K$  which is prime to $N$.
Write $\o=s\theta+t$ with $s,t\in\mathbb{Z}$.
Then, for $h\in\mathcal{F}_{N,\theta}$ we obtain
\begin{equation*}
h(\theta)^{\sigma([(\omega)])}=h^\alpha(\theta),
\end{equation*}
where $\alpha=\left[\begin{matrix}
t-B_\theta s& -C_\theta s\\
s& t
\end{matrix}\right]\in W_{N,\theta}$.
\end{proposition}
\begin{proof}
\cite[Theorem 6.31]{Shimura}.
\end{proof}

\par
We denote for convenience a quadratic form $aX^2+bXY+cY^2\in\mathbb{Z}[X,Y]$ by $[a,b,c]$.
Let $C(d_K)$ be the form class group of discriminant $d_K$.
Then we identify $C(d_K)$ with the set of all \textit{reduced quadratic forms}, namely
\begin{equation*}
C(d_K)=\left\{~[a,b,c]\in\mathbb{Z}[X,Y]~~\bigg|~
\begin{array}{c}
\textrm{$\gcd(a,b,c)=1$,~~$b^2-4ac=d_K$}\\
\textrm{$-a<b\leq a<c$ ~or~ $0\leq b\leq a=c$}
\end{array}\right\}
\end{equation*}
(\cite[Theorem 2.8]{Cox}).
Here we note that if $[a,b,c]\in C(d_K)$, then $a\leq\sqrt{{-d_K}/{3}}$ (\cite[p.29]{Cox}) and
 $C(d_K)$ is isomorphic to $\mathrm{Gal}(K_{(1)}/K)$ (\cite[Theorem 7.7]{Cox}).
For $Q=[a,b,c]\in C(d_K)$, we let
\begin{equation*}
\theta_Q=\frac{-b+\sqrt{d_K}}{2a}\in\mathbb{H}.
\end{equation*} 
Further, we define $\beta_Q=(\beta_p)_p\in\prod_p\mathrm{GL}_2(\mathbb{Z}_p)$ as follows:
\begin{itemize}
\item[] \textbf{Case 1} : $d_K\equiv 0 \pmod{4}$
\begin{equation*}
\beta_p=\left\{
\begin{array}{ll}
\left[\begin{matrix}
a&{b}/{2}\\
0&1
\end{matrix}\right]&\textrm{if $p\nmid a$}\vspace{0.2cm}\\
\left[\begin{matrix}
-{b}/{2}&-c\\
1&0
\end{matrix}\right]&\textrm{if $p\,|\,a$ and $p\nmid c$}\vspace{0.2cm}\\
\left[\begin{matrix}
-a-{b}/{2}&-c-{b}/{2}\\
1&-1
\end{matrix}\right]&\textrm{if $p\,|\,a$ and $p\,|\, c$},
\end{array}\right.
\end{equation*}
\item[] \textbf{Case 2} : $d_K\equiv 1 \pmod{4}$
\begin{equation*}
\beta_p=\left\{
\begin{array}{ll}
\left[\begin{matrix}
a&{(b-1)}/{2}\\
0&1
\end{matrix}\right]&\textrm{if $p\nmid a$}\vspace{0.2cm}\\
\left[\begin{matrix}
-{(b+1)}/{2}&-c\\
1&0
\end{matrix}\right]&\textrm{if $p\,|\,a$ and $p\nmid c$}\vspace{0.2cm}\\
\left[\begin{matrix}
-a-{(b+1)}/{2}&-c+{(1-b)}/{2}\\
1&-1
\end{matrix}\right]&\textrm{if $p\,|\,a$ and $p\,|\, c$}.
\end{array}\right.
\end{equation*}
\end{itemize}

\begin{proposition}\label{shimura reciprocity law}
For a positive integer $N$, we achieve a one-to-one correspondence
\begin{equation*}
\begin{array}{ccc}
W_{N,\theta}/\ker(\varphi_{N,\theta})\times C(d_K)&\longrightarrow&\mathrm{Gal}(K_{(N)}/K)\\
(\alpha,Q)&\longmapsto&(h(\theta)\mapsto h^{\widetilde{\alpha}\beta_Q}(\theta_Q))_{h\in\mathcal{F}_{N,\theta}}.
\end{array}
\end{equation*}
Here, $\widetilde{\alpha}$ is the preimage of $\alpha$ of the reduction 
\begin{equation*}
W_{N,\theta}/ \{\pm I_2\}\rightarrow W_{N,\theta}/\ker(\varphi_{N,\theta})
\end{equation*}
and the action of $\beta_Q$ on $\mathcal{F}_N$ is described by the action of $\beta\in\mathrm{GL}_2(\mathbb{Z}/N\mathbb{Z})/ \{\pm I_2\}$
so that $\beta\equiv \beta_p\pmod{N\mathbb{Z}_p}$ for all primes $p\,|\,N$.
\end{proposition}
\begin{proof}
It is immediate from Proposition \ref{ray class conjugate} and \cite[\S4]{Gee}.
\end{proof}

\section{Siegel-Ramachandra invariants}

In this section we shall introduce the arithmetic properties of Siegel functions and describe some necessary facts about Siegel-Ramachandra invariants for later use.
\par

For a rational vector $\mathbf{r}=\left[\begin{matrix}r_1\\r_2\end{matrix}\right]\in\mathbb{Q}^2\setminus\mathbb{Z}^2$, we define the \textit{Siegel function} $g_{\mathbf{r}}(\tau)$ on $\mathbb{H}$ by the following infinite product
\begin{equation*}
g_{\mathbf{r}}(\tau)=-q_{\tau}^{\frac{1}{2}\mathbf{B}_2(r_1)}e^{\pi i r_2(r_1-1)}(1-q_z)\prod_{n=1}^\infty(1-q_\tau^n q_z)(1-q_\tau^n q_z^{-1}),
\end{equation*}
where $\mathbf{B}_2(X)=X^2-X+1/6$ is the second Bernoulli polynomial and $z=r_1\tau+r_2$.
Then it has no zeros and poles on $\mathbb{H}$ (\cite{Siegel} or \cite[p.36]{Kubert}).
The following proposition describe the modularity criterion for Siegel functions.

\begin{proposition}\label{quad relation}
Let $N\geq 2$ be a positive integer and $\{m(\mathbf{r})\}_{\mathbf{r}=\left[\begin{smallmatrix}r_1\\r_2\end{smallmatrix}\right]\in\frac{1}{N}\mathbb{Z}^2\setminus\mathbb{Z}^2}$ be a family of integers such that $m(\mathbf{r})=0$ except for finitely many $\mathbf{r}$.
A finite product of Siegel functions
\begin{equation*}
\zeta\prod_{\mathbf{r}\in \frac{1}{N}\mathbb{Z}^2\setminus\mathbb{Z}^2}g_{\mathbf{r}}(\tau)^{m(\mathbf{r})}
\end{equation*}
belongs to $\mathcal{F}_N$ if 
\begin{equation*}
\begin{array}{l}
\displaystyle\sum_{\mathbf{r}}m(\mathbf{r})(Nr_1)^2\equiv\sum_{\mathbf{r}}m(\mathbf{r})(Nr_2)^2\equiv 0\pmod{\mathrm{gcd}(2,N)\cdot N},\vspace{0.2cm}\\
\displaystyle\sum_{\mathbf{r}}m(\mathbf{r})(Nr_1)(Nr_2)\equiv 0\pmod{N},\vspace{0.2cm}\\
\displaystyle\sum_{\mathbf{r}}m(\mathbf{r})\cdot\mathrm{gcd}(12,N)\equiv 0\pmod{12}.
\end{array}
\end{equation*}
Here, 
\begin{equation*}
\zeta=\prod_{\mathbf{r}}e^{\pi i r_2(1-r_1)m(\mathbf{r})}\in\mathbb{Q}(\zeta_{2N^2}).
\end{equation*}
\end{proposition}
\begin{proof}
\cite[Chapter 3, Theorem 5.2 and 5.3]{Kubert}.
\end{proof}

\begin{proposition}\label{Siegel property}
Let $\mathbf{r}, \mathbf{s}\in({1}/{N})\mathbb{Z}^2\setminus\mathbb{Z}^2$ for a positive integer $N\geq 2$. 
\begin{itemize}
\item[\textup{(i)}] $g_{\mathbf{r}}(\tau)^{12N}$ satisfies the relation
\begin{equation*}
g_{\mathbf{r}}(\tau)^{12N}=g_{-\mathbf{r}}(\tau)^{12N}=g_{\langle\mathbf{r}\rangle}(\tau)^{12N}
\end{equation*}
where $\langle X\rangle$ is the fractional part of $X\in\mathbb{R}$ such that $0\leq \langle X\rangle<1$ and
$\langle\mathbf{r}\rangle=\left[\begin{matrix}\langle r_1\rangle\\ \langle r_2\rangle\end{matrix}\right]$.
\item[\textup{(ii)}] $g_{\mathbf{r}}(\tau)^{12N}$ belongs to $\mathcal{F}_N$.
Moreover, $\alpha\in \mathrm{GL}_2(\mathbb{Z}/N\mathbb{Z})/\{\pm I_2\}$ acts on it by
\begin{equation*}
\left(g_{\mathbf{r}}(\tau)^{12N}\right)^\alpha=g_{{^t}\alpha\mathbf{r}}(\tau)^{12N}.
\end{equation*}
\item[\textup{(iii)}] For $\gamma=\left[\begin{matrix}
a&b\\
c&d
\end{matrix}\right]\in\mathrm{SL}_2(\mathbb{Z})$ we have
\begin{equation*}
\left(\frac{g_{\mathbf{r}}(\tau)}{g_{\mathbf{s}}(\tau)}\right)\circ\gamma=\frac{g_{{}^t\gamma\mathbf{r}}(\tau)}{g_{{}^t\gamma\mathbf{s}}(\tau)}.
\end{equation*}

\end{itemize}
\end{proposition}
\begin{proof}
\cite[Proposition 1.1]{Jung} and \cite[p.27, K 1]{Kubert}.
\end{proof}

Let $L$ be a lattice in $\mathbb{C}$ and $t\in\mathbb{C}\setminus L$ be a point of finite order with respect to $L$.
We choose a basis $[\o_1,\o_2]$ of $L$ such that $z={\o_1}/{\o_2}\in\mathbb{H}$ and write
\begin{equation*}
t=r_1\o_1+r_2\o_2
\end{equation*}
for some $\left[\begin{matrix}r_1\\r_2\end{matrix}\right]\in\mathbb{Q}^2\setminus\mathbb{Z}^2$.
And, we define a function
\begin{equation*}
g(t,[\o_1,\o_2])=g_{\left[\begin{smallmatrix}r_1\\r_2\end{smallmatrix}\right]}(z),
\end{equation*}
which depends on the choice of $\o_1, \o_2$.
However, by raising $12$-th power we obtain a function $g^{12}(t,L)$ of $t$ and $L$ (\cite[p.31]{Kubert}).
\par
Now let $K$ be an imaginary quadratic field of discriminant $d_K$, $\F$ be a nontrivial proper integral ideal of $K$ and $N$ be the smallest positive integer in $\F$.
For $C\in\mathrm{Cl}(\F)$, we define the \textit{Siegel-Ramachandra invariant} of conductor $\F$ at $C$  by 
\begin{equation*}
g_\F(C)=g^{12N}(1,\F\mathfrak{c}^{-1}),
\end{equation*}
where $\mathfrak{c}$ is any integral ideal in $C$.
This value depends only on $\F$ and the class $C$, not on the choice of $\mathfrak{c}$.
\par

\begin{proposition}\label{Galois action}
Let $C, C'\in\mathrm{Cl}(\F)$ with $\F\neq \O_K$.
\begin{itemize}
\item[\textup{(i)}] $g_\F(C)$ lies in $K_\F$ as an algebraic integer. 
If $N$ is composite, $g_\F(C)$ is a unit in $K_\F$.
\item[\textup{(ii)}]
We have the transformation formula
\begin{equation*}
g_\F(C)^{\sigma(C')}=g_\F(CC'),
\end{equation*}
where $\sigma$ is the Artin map stated in \textup{(\ref{artin map})}.
\item[\textup{(iii)}]
${g_\F(C')}/{g_\F(C)}$ is a unit in $K_\F$.
\end{itemize}
\end{proposition}
\begin{proof}
\cite[Chapter 19, Theorem 3]{Lang} and \cite[Chapter 11, Theorem 1.2]{Kubert}.
\end{proof}

Further, we let $\mathfrak{a}$ be an integral ideal of $K$ which is not divisible by $\F$.
For a class $C\in\mathrm{Cl}(\F)$, we define the \textit{Robert invariant} by
\begin{equation*}
u_\mathfrak{a}(C)=\frac{g^{12}(1,\F\mathfrak{c}^{-1})^{\mathcal{N}(\mathfrak{a})}}{g^{12}(1,\F\mathfrak{a}^{-1}\mathfrak{c}^{-1})},
\end{equation*}
where $\mathfrak{c}$ is any integral ideal in $C$.
This depends only on the class $C$.
Note that $u_\mathfrak{a}(C)$ belongs to $K_\F$, but it is not necessarily a unit (\cite[Chapter 11, Theorem 4.1]{Kubert}).
We shall take products of such invariants with a linear condition in order to get units. 
Let $C_0$ be the unit class in Cl$(\F)$ and $\mathfrak{R}_\F^*$ be the group of all finite products
\begin{equation*}
\prod_{\mathfrak{a}} u_\mathfrak{a}(C_0)^{m(\mathfrak{a})}\quad(m(\mathfrak{a})\in\mathbb{Z})
\end{equation*}
taken with all integral ideal $\mathfrak{a}$ of $K$ prime to $6N$ satisfying $\sum_\mathfrak{a} m(\mathfrak{a})(\mathcal{N}(\mathfrak{a})-1)=0$.
Then $\mathfrak{R}_\F^*$ becomes a subgroup of the units in $K_\F$ (\cite[Chapter 11, Theorem 4.2]{Kubert}).
\par
Let $\o_{K_\F}$ be the number of roots of unity in $K_\F$ and define the group
\begin{equation*}
\Phi_\F(\o_{K_\F})=\left\{ \prod_{C\in\mathrm{Cl}(\F)}g_\F(C)^{n(C)}~\Big|~\sum_C n(C)=0~\textrm{and}~\sum_C n(C)\mathcal{N}(\mathfrak{a}_C)\equiv 0 \pmod{\o_{K_\F}}\right\},
\end{equation*}
where $\mathfrak{a}_C$ is any integral ideal in $C$ prime to $6N$ and the exponents $n(C)$ are integers.
The value $\mathcal{N}(\mathfrak{a}_C) \pmod{\o_{K_\F}}$ does not depend on the choice of $\mathfrak{a}_C$ (\cite[Chapter 9, Lemma 4.1]{Kubert}), and so the group is well-defined. 

\begin{proposition}\label{unit group}
We have
\begin{equation*}
(\mathfrak{R}_\F^*)^N=\Phi_\F(\o_{K_\F}).
\end{equation*}
\end{proposition}
\begin{proof}
\cite[Chapter 11, Theorem 4.3]{Kubert}
\end{proof}

\begin{lemma}\label{Nth root}
Let $\F=N\O_K$ for a positive integer $N\geq 2$ and $C'$ be an element of $\mathrm{Cl}(\F)$ satisfying $\mathcal{N}(\mathfrak{a}_{C'})\equiv 1\pmod{\o_{K_\F}}$ for some integral ideal $\mathfrak{a}_{C'}$ in $C'$ prime to $6N$.
Then any $N$-th root of the value ${g_\F(C')}/{g_\F(C_0)}$ is a unit in $K_\F$.
\end{lemma}
\begin{proof}
Observe that ${g_\F(C')}/{g_\F(C_0)}\in \Phi_\F(\o_{K_\F})$.
It then follows from Proposition \ref{unit group} that 
\begin{equation*}
\frac{g_\F(C')}{g_\F(C_0)}=u^N
\end{equation*}
for some $u\in \mathfrak{R}_\F^*$.
Since $u$ is a unit in $K_\F$ and $\zeta_N\in\mathbb{Q}(\zeta_N)\subset K_{(N)}= K_\F$, any $N$-th root of the value ${g_\F(C')}/{g_\F(C_0)}$ is also a unit in $K_\F$.
\end{proof}

Let $\mathfrak{f}$ be a nontrivial integral ideal of $K$ and $\chi$ be a character of $\mathrm{Cl}(\F)$.
The conductor of $\chi$ is defined by the largest ideal $\mathfrak{g}$ dividing $\F$ such that $\chi$ is obtained by a composition of a character of $\mathrm{Cl}(\mathfrak{g})$ and the natural homomorphism $\mathrm{Cl}(\F)\rightarrow \mathrm{Cl}(\mathfrak{g})$, and we denote it by $\F_\chi$.
Similarly, if $\chi'$ is a character of $(\O_K /\mathfrak{f})^\times$ then 
we define the conductor of $\chi'$  by the largest ideal $\mathfrak{g}$ dividing $\F$ for which $\chi'$ is induced by a composition of a character of $(\O_K /\mathfrak{g})^\times$ with the natural homomorphism $(\O_K /\mathfrak{f})^\times\rightarrow (\O_K /\mathfrak{g})^\times$, and denote it by $\F_{\chi'}$.
The map 
\begin{equation}\label{map}
\begin{array}{ccc}
(\O_K /\mathfrak{f})^\times&\longrightarrow &\mathrm{Cl}(\F)\\
\alpha +\F&\longmapsto& \left[(\alpha)\right]
\end{array}
\end{equation}
is a well-defined homomorphism whose kernel is
\begin{equation*}
\{\alpha+\F\in (\O_K/\F)^\times ~|~\alpha\in \O_K^\times \} .
\end{equation*}
For a character $\chi$ of $\mathrm{Cl}(\F)$, we derive a character $\chi'$ of $(\O_K /\mathfrak{f})^\times$ by composing with the map (\ref{map}). 
Then, by definition $\F_\chi=\F_{\chi'}$ is immediate.

\par

Now let $\chi$ be a nontrivial character of $\mathrm{Cl}(\F)$ with $\F\neq\O_K$ and $\chi_0$ be the primitive character of $\mathrm{Cl}(\F_\chi)$ corresponding to $\chi$.
We define the \textit{Stickelberger element} and the \textit{L-function} for $\chi$ by
\begin{eqnarray*}
S_\F(\chi,g_\F)&=&\sum_{C\in\mathrm{Cl}(\F)}\chi(C)\log|g_\F(C)|,\\
L_\F(s,\chi)&=&\sum_{\substack{(0)\neq \mathfrak{a}\subset \O_K \\ \gcd(\mathfrak{a},\F)=1 }} \frac{\chi(\mathfrak{a})}{\mathcal{N}(\mathfrak{a})^s}\quad (s\in\mathbb{C}),
\end{eqnarray*}
respectively.
The second Kronecker limit formula explains the relation between the Stickelberger element and the L-function as follows:

\begin{proposition}\label{L-function relation}
Let $\chi$ be a nontrivial character of $\mathrm{Cl}(\F)$ with $\F_\chi\neq\O_K$.
Then we have
\begin{equation*}
L_{\F_\chi}(1,\chi_0)\prod_{\substack{\P\,|\,\F \\ \P\,\nmid\,\F_\chi }}(1-\overline{\chi_0}([\P]))
=-\frac{2\pi \chi_0([\gamma\mathfrak{d}_K \F_\chi])}{6N(\F_\chi) \o(\F_\chi)T_\gamma(\overline{\chi_0})\sqrt{-d_K}}  \cdot S_\F(\overline{\chi},g_\F),
\end{equation*}
where $\mathfrak{d}_K$ is the different of $K/\mathbb{Q}$, $\gamma$ is an element of $K$ such that $\gamma\mathfrak{d}_K\F_\chi$ is an integral ideal of $K$ prime to $\F_\chi$,
$N(\F_\chi)$ is the smallest positive integer in $\F_\chi$,
$\o(\F_\chi)$ is the number of roots of unity in $K$ which are $\equiv 1\pmod{\F_\chi}$ and
\begin{equation*}
T_\gamma(\overline{\chi_0})=\sum_{x+\F_\chi\in(\O_K/\F_\chi)^\times}\overline{\chi_0}([x\O_K])e^{2\pi i \mathrm{Tr}_{K/\mathbb{Q}}(\gamma x)}.
\end{equation*}

\end{proposition}
\begin{proof}
\cite[Chapter 11 \S2, LF 2]{Kubert}.
\end{proof}

\begin{remark}\label{Stickremark}
 Since $\chi_0$ is a nontrivial character of $\mathrm{Cl}(\mathfrak{f}_\chi)$, we obtain $L_{\mathfrak{f}_\chi}(1,\chi_0)\neq0$ (\cite[Chapter V, Theorem 10.2]{Janusz}).
Furthermore, the Gauss sum $T_\gamma(\overline{\chi}_0)$ is also nonzero (\cite[Chapter 22 $\S$1, G 3]{Lang}).
If every prime ideal factor of $\mathfrak{f}$ divides $\mathfrak{f}_\chi$ then we understand the Euler factor 
$\prod_{{\P\,|\,\F,~  \P\,\nmid\,\F_\chi }}(1-\overline{\chi_0}([\P]))$ to be $1$, and hence
we conclude $S_\F(\overline{\chi},g_\F)\neq 0$.
\end{remark}

\section{Generation of ray class fields by Siegel-Ramachandra invariants}\label{Generation of ray class fields by Siegel-Ramachandra invariants}

We shall improve in this section the result of Schertz \cite{Schertz} concerning construction of ray class fields by means of
Siegel-Ramachandra invariants.
\par
Let $K$ be an imaginary quadratic field of discriminant $d_K$ and $\o_K$ be the number of roots of unity in $K$.
For a nontrivial integral ideal $\F$ of $K$, we regard $\o(\F)$ as the number of roots of unity in $K$ which are $\equiv 1\pmod{\F}$.
Let $\phi$ be the Euler function for ideals, namely
\begin{equation}\label{Euler function}
\phi(\F)=\left|(\O_K/\F)^\times\right|=\mathcal{N}(\F)\displaystyle\prod_{\substack{\mathfrak{p}\,|\, \F\\ \mathfrak{p}~\textrm{prime} }}\left(1-\frac{1}{\mathcal{N}(\mathfrak{p})}\right).
\end{equation}

\begin{proposition}\label{order of ray class field}
If $\F$ is a nontrivial integral ideal of $K$, then we get
\begin{equation*}
[K_\F :K]=h_K \phi(\F)\frac{\o(\F)}{\o_K}
\end{equation*}
where $h_K$ is the class number of $K$.
\end{proposition}
\begin{proof}
\cite[Chapter VI, Theorem 1]{Lang2}.
\end{proof}

\begin{lemma}\label{extension}
Let $H\subset G$ be two finite abelian groups, $g\in G\setminus H$ and $n$ be the order of the coset $[g]$ in $G/H$.
Then for any character $\chi$ of $H$, we can extend it to a character $\psi$ of $G$ such that $\psi(g)$ is any fixed $n$-th root of $\chi(g^n)$.

\end{lemma}
\begin{proof}
\cite[Chapter VI, Proposition 1]{Serre}.
\end{proof}

Now, let $\F$ be a nontrivial proper integral ideal of $K$ with prime ideal factorization 
\begin{equation*}
\F=\prod_{i=1}^r \mathfrak{p}_i^{n_i}
\end{equation*}
and $C_0$ be the unit class in Cl$(\F)$.
Consider the quotient group 
\begin{equation*}
\mathbf{G}=(\O_K/\mathfrak{f})^\times  / \{\alpha+\mathfrak{f}\in (\O_K/\mathfrak{f})^\times~|~\alpha\in \O_K^\times \}.
\end{equation*}
Then one can view $\mathbf{G}$ as a subgroup of $\mathrm{Cl}(\F)$ via the map (\ref{map}).
For each $i$, we set
\begin{equation*}
\mathbf{G}_i=(\O_K/\mathfrak{p}_i^{n_i})^\times  / \{\alpha+\mathfrak{p}_i^{n_i}\in (\O_K/\mathfrak{p}_i^{n_i})^\times~|~\alpha\in \O_K^\times \}.
\end{equation*}

\begin{proposition}\label{existence of character}
Assume that $|\mathbf{G}_i|>2$ for every $i$.
Then, for any class $C ~(\neq C_0)\in$ $\mathrm{Cl}(\F)$ there exists a character $\chi$ of $\mathrm{Cl}(\F)$ such that $\chi(C)\neq 1$ and $\mathfrak{p}_i\, |\, \mathfrak{f}_\chi$ for all $i$.
\end{proposition}
\begin{proof}
Since $C\neq C_0$, there is a character $\chi$ of Cl$(\F)$ such that $\chi(C)\neq 1$.
We set $n$ to be the order of $C$ in the quotient group $\mathrm{Cl}(\F)/\mathbf{G}$.
Then $C^n=\left[(\beta)\right]$ for some $\beta\in\O_K$ which is relatively prime to $\F$.
Suppose that $\mathfrak{f}_\chi$ is not divided by $\mathfrak{p}_i$ for some $i$.
If $\mathbf{G}_i\neq\langle\beta+\P_i^{n_i}\rangle$ then we can find a nontrivial character $\psi$ of $(\O_K/\mathfrak{p}_i^{n_i})^\times$ for which $\psi$ is trivial on
$\{\alpha+{\mathfrak{p}_i^{n_i}}\in (\O_K/\mathfrak{p}_i^{n_i})^\times~|~\alpha\in \O_K^\times \}$ and $\psi(\beta+{\mathfrak{p}_i^{n_i}})=1$.
By composing with a natural homomorphism $(\O_K /\mathfrak{f})^\times\rightarrow (\O_K /\mathfrak{p}_i^{n_i})^\times$, 
we are able to extend $\psi$ to a character $\psi'$ of $(\O_K/\F)^\times$ whose conductor is divisible only by $\mathfrak{p}_i$.
Observe that $\psi'$ is trivial on $ \{\alpha+{\mathfrak{f}}\in (\O_K/\mathfrak{f})^\times~|~\alpha\in \O_K^\times \}$ and so $\psi'$ becomes a character of $\mathbf{G}$.
It then follows from Lemma \ref{extension} that we can also extend $\psi'$ to a character $\psi''$ of $\mathrm{Cl}(\F)$ such that $\psi''(C)=1$.
Now assume that $\mathbf{G}_i=\langle\beta+{\P_i^{n_i}}\rangle$.
Since $|\mathbf{G}_i|>2$, there is a nontrivial character $\psi$ of $(\O_K/\mathfrak{p}_i^{n_i})^\times$ so that $\psi$ is trivial on
$\{\alpha+{\mathfrak{p}_i^{n_i}}\in (\O_K/\mathfrak{p}_i^{n_i})^\times~|~\alpha\in \O_K^\times \}$ and $\psi(\beta+{\mathfrak{p}_i^{n_i}})\neq 1, \chi(C^n)^{-1}$.
In a similar way, one can extend $\psi$ to a character $\psi''$ of $\mathrm{Cl}(\F)$ in such a way that $\F_{\psi''}$ is divisible only by $\P_i$ and $\psi''(C)\neq \chi(C)^{-1}$.
Therefore, the character $\chi\psi''$ of $\mathrm{Cl}(\F)$ satisfies $\chi\psi''(C)\neq 1$, 
$\mathfrak{p}_i\, |\, \mathfrak{f}_{\chi\psi''}$ and $\mathfrak{f}_\chi\, |\, \mathfrak{f}_{\chi\psi''}$ in both cases.
By continuing this process for all $i$, we achieve a desired character in the end.

\end{proof}

\begin{lemma}\label{order condition}
With the notations as above,  $|\mathbf{G}_i|=1$ or $2$ if and only if $\mathfrak{p}_i^{n_i}$ satisfies one of the following conditions:
\begin{itemize}
\item[] \textbf{Case 1} : $K\neq \mathbb{Q}(\sqrt{-1}), \mathbb{Q}(\sqrt{-3})$ 
\begin{itemize}
\item[$\bullet$] $2$ is not inert in $K$, $\P_i$ is lying over $2$ and $n_i=1,2$ or $3$.
\item[$\bullet$] $3$ is not inert in $K$, $\P_i$ is lying over $3$ and $n_i=1$.
\item[$\bullet$] $5$ is not inert in $K$, $\P_i$ is lying over $5$ and $n_i=1$.
\end{itemize}
\item[] \textbf{Case 2} : $K=\mathbb{Q}(\sqrt{-1})$ 
\begin{itemize}
\item[$\bullet$] $\P_i$ is lying over $2$ and $n_i=1,2,3$ or $4$.
\item[$\bullet$] $\P_i$ is lying over $3$ and $n_i=1$.
\item[$\bullet$] $\P_i$ is lying over $5$ and $n_i=1$.
\end{itemize}
\item[] \textbf{Case 3} : $K=\mathbb{Q}(\sqrt{-3})$ 
\begin{itemize}
\item[$\bullet$] $\P_i$ is lying over $2$ and $n_i=1$ or $2$.
\item[$\bullet$] $\P_i$ is lying over $3$ and $n_i=1$ or $2$.
\item[$\bullet$] $\P_i$ is lying over $7$ and $n_i=1$.
\item[$\bullet$] $\P_i$ is lying over $13$ and $n_i=1$.
\end{itemize}
\end{itemize}
\end{lemma}

\begin{proof}
First, we note that 
\begin{equation}\label{order of G}
|\mathbf{G}_i|=\phi(\P_i^{n_i})\frac{\omega(\P_i^{n_i})}{\omega_K}.
\end{equation}
Let $p_i$ be a prime number such that $\P_i$ is lying over $p_i$.
It then follows from (\ref{Euler function}) that
\begin{equation}\label{order of phi}
\phi(\P_i^{n_i})=\left\{
\begin{array}{ll}
p_i^{2n_i}-p_i^{2n_i-2} & \textrm{if $p_i$ is inert in $K$}\vspace{0.2cm}\\
p_i^{n_i}-p_i^{n_i-1} & \textrm{otherwise}.
\end{array}
\right.
\end{equation}
If $K\neq \mathbb{Q}(\sqrt{-1}), \mathbb{Q}(\sqrt{-3})$, then $\o_K=2$ and
\begin{equation}\label{w1}
\omega(\P_i^{n_i})=\left\{
\begin{array}{ll}
2 & \textrm{if $\P_i^{n_i}\,|\,2\O_K$}\vspace{0.1cm}\\
1 & \textrm{otherwise}.
\end{array}
\right.
\end{equation}
If $K=\mathbb{Q}(\sqrt{-1})$, then $\o_K=4$ and 
\begin{equation}\label{discriminant}
\left(\frac{d_K}{p_i}\right)=\left(\frac{-4}{p_i}\right)=\left\{
\begin{array}{ll}
1 & \textrm{if $p_i\equiv 1\pmod{4}$}\vspace{0.1cm}\\
-1 & \textrm{if $p_i\equiv 3\pmod{4}$}\vspace{0.1cm}\\
0  & \textrm{if $p_i=2$}
\end{array}
\right.
\end{equation}
where $(\frac{d_K}{p_i})$ stands for the Kronecker symbol.
Since $2\O_K=(1-\sqrt{-1})^2\O_K$, we deduce
\begin{equation}\label{w2}
\omega(\P_i^{n_i})=\left\{
\begin{array}{ll}
4 & \textrm{if $\P_i^{n_i}=(1-\sqrt{-1})\O_K$}\vspace{0.1cm}\\
2 & \textrm{if $\P_i^{n_i}\neq(1-\sqrt{-1})\O_K$ and $\P_i^{n_i}\,|\,2\O_K$} \vspace{0.1cm}\\
1 & \textrm{otherwise}.
\end{array}
\right.
\end{equation}
If $K=\mathbb{Q}(\sqrt{-3})$, then $\o_K=6$ and 
\begin{equation}\label{discriminant2}
\left(\frac{d_K}{p_i}\right)=\left(\frac{-3}{p_i}\right)=\left\{
\begin{array}{ll}
1 & \textrm{if $p_i\equiv 1, 7\pmod{12}$}\vspace{0.1cm}\\
-1 & \textrm{if $p_i\equiv 5, 11\pmod{12}$ or $p_i=2$} \vspace{0.1cm}\\
0  & \textrm{if $p_i=3$}.
\end{array}
\right.
\end{equation}
Here we observe that 2 is inert in $K$ and 3 is ramified in $K$, so to speak, 
$3\O_K=\left(\frac{3+\sqrt{-3}}{2}\right)^2\O_K$.
One can then readily show that
\begin{equation}\label{w3}
\omega(\P_i^{n_i})=\left\{
\begin{array}{ll}
3 & \textrm{if $\P_i^{n_i}=\left(\displaystyle\frac{3+\sqrt{-3}}{2}\right)\O_K$}\vspace{0.1cm}\\
2 & \textrm{if $\P_i^{n_i}=2\O_K$} \vspace{0.1cm}\\
1 & \textrm{otherwise}.
\end{array}
\right.
\end{equation}
Therefore, the lemma follows from (\ref{order of G})$\sim$(\ref{w3}).
\end{proof}

\begin{remark}
If 2 is not inert in $K$, $\P_i$ is lying over 2 and $n_i=1$, then $K_\F=K_{\F\P_i^{-n_i}}$.
Indeed, we see from Proposition \ref{order of ray class field} that
\begin{equation}\label{ray class formula}
[K_{\F\P_i^{-n_i}} :K]= \frac{\o(\F\P_i^{-n_i})}{\phi(\P_i^{n_i})\o(\F)}\cdot[K_\F :K].
\end{equation}
Since $\phi(\P_i^{n_i})=1$ and $\o(\F\P_i^{-n_i})=\o(\F)$ in this case, we obtain the conclusion.
\end{remark}

From now on we assume that $K_\F\neq K_{\F\P_i^{-n_i}}$ for every $i$ and let $\mathbf{N}_\F$ be the number of $i$ such that $|G_i|=1$ or $2$.
After reordering prime ideal factors of $\F$ if necessary, we may suppose that $|G_{i}|=1$ or $2$ for $i=1,2,\ldots, \mathbf{N}_\F$.
For any intermediate field $F$ of the extension $K_\mathfrak{f}/K$ we mean
by $\mathrm{Cl}(K_\mathfrak{f}/F)$
the subgroup of $\mathrm{Cl}(\mathfrak{f})$ corresponding to $\mathrm{Gal}(K_\mathfrak{f}/F)$ via the Artin map $\sigma$ stated in (\ref{artin map}).

\begin{theorem}\label{main theorem}
Let $\F$ be a nontrivial proper integral ideal of $K$ with prime ideal factorization $\F=\prod_{i=1}^r \mathfrak{p}_i^{n_i}$
such that $|\mathbf{G}_i|=1$ or $2$ for $i=1,2,\ldots, \mathbf{N}_\F$.
Assume that $K_\F\neq K_{\F\P_i^{-n_i}}$ for every $i$ and 
\begin{equation}\label{bound-assumption}
\sum_{i=1}^{\mathbf{N}_\F}\frac{1}{\phi(\P_i^{n_i})}\leq \frac{1}{2}.
\end{equation}
Then for any class $C\in\mathrm{Cl}(\F)$ and any nonzero integer $n$, we get
\begin{equation*}
K_\F=K\big(g_\F(C)^n\big).
\end{equation*}
\end{theorem}
\begin{proof}
Let $F=K(g_\F(C_0)^n)$.
On the contrary suppose that $F$ is properly contained in $K_\F$, that is, $\mathrm{Cl}(K_\F/F)\neq\{1\}$.
Then we claim that there exists a character $\chi$ of $\mathrm{Cl}(\F)$ satisfying $\chi|_{\mathrm{Cl}(K_\F/F)}\neq 1$ and $\mathfrak{p}_i\, |\, \mathfrak{f}_\chi$ for $i=1,2,\ldots, \mathbf{N}_\F$.
Indeed, we deduce that
\begin{equation}\label{character1}
\begin{array}{lll}
M_1&=&\left|\left\{\textrm{characters $\chi$ of $\mathrm{Cl}(\F)~|~ \chi|_{\mathrm{Cl}(K_\F/F)}\neq 1$} \right\}\right|\vspace{0.2cm}\\
&=&\left|\left\{\textrm{characters $\chi$ of $\mathrm{Cl}(\F)$} \right\}\right|
-\left|\left\{\textrm{characters $\chi$ of $\mathrm{Cl}(\F)~|~ \chi|_{\mathrm{Cl}(K_\F/F)}=1$} \right\}\right|\vspace{0.2cm}\\
&=&[K_\F : K]-[F:K]\vspace{0.2cm}\\
&=&[K_\F : K]\left(1-\displaystyle\frac{1}{[K_\F:F]}\right)\vspace{0.2cm}\\
&\geq& \displaystyle\frac{1}{2}[K_\F : K].
\end{array}
\end{equation}
Thus if $\mathbf{N}_\F=0$,  the claim is clear.
Observe that a trivial character is not contained in the set stated in (\ref{character1}).
Now assume $\mathbf{N}_\F\geq 1$ and let 
\begin{equation*}
M_2=
\Big|\left\{\textrm{characters $\chi\neq 1$ of $\mathrm{Cl}(\F)~|~ \P_{i}\nmid\F_{\chi}$ for some $i\in\{1,2,\ldots,\mathbf{N}_\F\}$} \right\}\Big|.
\end{equation*}
First, we suppose $\mathbf{N}_\F\neq 2$.
Then we derive that
\begin{equation*}
\begin{array}{lll}
M_2&=& \left|\left\{\textrm{characters $\chi$ of $\mathrm{Cl}(\F)~|~ \P_{i}\nmid\F_{\chi}$ for some $i\in\{1,2,\ldots,\mathbf{N}_\F\}$} \right\}\right|-1 \vspace{0.2cm} \\
&=&\left|\left\{\textrm{characters $\chi$ of $\mathrm{Cl}(\F)~|~ \F_{\chi}\,|\,\F\P_{i}^{-n_{i}}$ for some $i\in\{1,2,\ldots,\mathbf{N}_\F\}$  } \right\}\right|-1\vspace{0.2cm}\\
&\leq&\displaystyle\sum_{i=1}^{\mathbf{N}_\F}\left|\left\{\textrm{characters $\chi$ of $\mathrm{Cl}(\F\P_{i}^{-n_{i}})$} \right\}\right|-1 \vspace{0.2cm}\\
&=&\displaystyle\sum_{i=1}^{\mathbf{N}_\F}[K_{\F\P_{i}^{-n_{i}}}:K]-1 \vspace{0.2cm}\\
&=&\displaystyle\left(\sum_{i=1}^{\mathbf{N}_\F}\frac{\o(\F\P_i^{-n_i})}{\phi(\P_i^{n_i})\o(\F)}\right)\left[K_{\F}:K\right]-1 \quad\textrm{(by (\ref{ray class formula}))}.
\end{array}
\end{equation*}
If $\mathbf{N}_\F=1$, then $M_2\leq ({1}/{2})[K_\F:K]-1$ since $K_\F\neq K_{\F\P_i^{-n_i}}$ for every $i$.
And, if $\mathbf{N}_\F\geq 3$, then $\o(\F)=\o(\F\P_i^{-n_i})=1$ for every $i$ because $\F$ has at least three prime ideal factors.
Hence we attain $M_2\leq ({1}/{2})[K_\F:K]-1$ again by the assumption (\ref{bound-assumption}).
Now assume $\mathbf{N}_\F=2$.
Since $\F$ has at least two prime ideal factors, we achieve that
\begin{equation*}
\begin{array}{lll}
M_2&=&\left|\left\{\textrm{characters $\chi$ of $\mathrm{Cl}(\F)~|~ \F_{\chi}\,|\,\F\P_{i}^{-n_{i}}$ for some $i\in\{1,2\}$  } \right\}\right|-1\vspace{0.2cm}\\
&=&\displaystyle\sum_{i=1}^{2}\left|\left\{\textrm{characters $\chi$ of $\mathrm{Cl}(\F\P_{i}^{-n_{i}})$} \right\}\right|
-\left|\left\{\textrm{characters $\chi$ of $\mathrm{Cl}(\F\P_{1}^{-n_{1}}\P_{2}^{-n_{2}})$} \right\}\right|-1\vspace{0.2cm}\\
&=&\displaystyle\sum_{i=1}^{2}[K_{\F\P_{i}^{-n_{i}}}:K]-[K_{\F\P_{1}^{-n_{1}}\P_{2}^{-n_{2}} }:K]-1\vspace{0.2cm}\\
&=&\displaystyle\left(\sum_{i=1}^{2}\frac{\o(\F\P_i^{-n_i})}{\phi(\P_i^{n_i})}-\frac{\o(\F\P_1^{-n_1}\P_2^{-n_2})}{ \phi(\P_1^{n_1})\phi(\P_2^{n_2})} \right)\left[K_{\F}:K\right]-1.\vspace{0.2cm}
\end{array}
\end{equation*}
Here we note that $\o(\F\P_i^{-n_i})\neq 1$ occurs only when $\F=\P_1^{n_1}\P_2^{n_2}$, and $\o(\F\P_1^{-n_1}\P_2^{-n_2})=\o_K$ in this case.
Using this fact one can check that if the assumption (\ref{bound-assumption}) holds then
\begin{equation}\label{inequality}
\left(\sum_{i=1}^{2}\frac{\o(\F\P_i^{-n_i})}{\phi(\P_i^{n_i})}-\frac{\o(\F\P_1^{-n_1}\P_2^{-n_2})}{ \phi(\P_1^{n_1})\phi(\P_2^{n_2})} \right)\leq\frac{1}{2},
\end{equation}
which yields $M_2\leq ({1}/{2})[K_\F:K]-1$.
Thus for any $\mathbf{N}_\F\geq 1$ we have $M_1>M_2$ and so the claim is proved.
Furthermore, the proof of Proposition \ref{existence of character} shows that there is a character $\psi''$ of $\mathrm{Cl}(\F)$ for which $\chi\psi''|_{\mathrm{Cl}(K_\F/F)}\neq 1$, $\mathfrak{f}_{\chi}\, |\, \mathfrak{f}_{\chi\psi''}$ and $\mathfrak{p}_i \,|\, \mathfrak{f}_{\chi\psi''}$ for all $i$.
So, by replacing $\chi$ by $\chi\psi''$, we get a character $\chi$ of $\mathrm{Cl}(\F)$ satisfying $\chi|_{\mathrm{Cl}(K_\F/F)}\neq 1$ and $\mathfrak{p}_i\, |\, \mathfrak{f}_\chi$ for all $i$.
\par
Since $\chi$ is nontrivial and $\F_\chi\neq\O_K$, we have $S_\F(\overline{\chi},g_\F)\neq 0$ by Proposition \ref{L-function relation}.
On the other hand, we induce that
\begin{eqnarray*}
S_\F(\overline{\chi},g_\F)&=&\frac{1}{n}\sum_{C\in\mathrm{Cl}(\F)}\overline{\chi}(C)\log|g_\F(C)^n|\\
&=&\frac{1}{n}\sum_{C\in\mathrm{Cl}(\F)}\overline{\chi}(C)\log\left|(g_\F(C_0)^n)^{\sigma(C)}\right|\quad\textrm{(by Proposition \ref{Galois action})}\\
&=&\frac{1}{n}\sum_{[C_1]\in\mathrm{Cl}(\F)/\mathrm{Cl}(K_\F/F)}\left(\sum_{C_2\in\mathrm{Cl}(K_\F/F)}\overline{\chi}(C_1C_2)\log\left|(g_\F(C_0)^n)^{\sigma(C_1C_2)}\right|\right)\\
&=&\frac{1}{n}\sum_{[C_1]\in\mathrm{Cl}(\F)/\mathrm{Cl}(K_\F/F)}\overline{\chi}(C_1)\log\left|(g_\F(C_0)^n)^{\sigma(C_1)}\right|\left(\sum_{C_2\in\mathrm{Cl}(K_\F/F)}\overline{\chi}(C_2)\right)\\
&=&0,
\end{eqnarray*}
because $g_\F(C_0)^n\in F$ and $\chi|_{\mathrm{Cl}(K_\F/F)}\neq 1$.
It gives a contradiction.
Therefore, $F=K_\F$ and so $K_\F=K\left(g_\F(C)^n\right)$ for any $C\in\mathrm{Cl}(\F)$ since $g_\F(C)^n=(g_\F(C_0)^n)^{\sigma(C)}$ by Proposition \ref{Galois action}.
\end{proof}

\begin{remark}
\textup{(i)} When $\mathbf{N}_\F=0$ or $1$, the assumption (\ref{bound-assumption}) is always satisfied.
\par
\textup{(ii)}
If $\mathbf{N}_\F=2$ and $\F=\P_1^{n_1}\P_2^{n_2}$, then one can show that the inequality (\ref{inequality}) holds except the case where $K\neq\mathbb{Q}(\sqrt{-1}),\mathbb{Q}(\sqrt{-3})$, 2 is ramified in $K$ and $\F=\P_{(2)}^2\P_{(5)}$.
Here, $\P_{(p)}$ stands for a prime ideal of $K$ lying over a prime number $p$.
Therefore we are able to establish Theorem \ref{main theorem} again under the above condition.
\end{remark}

\section{Ray class fields by smaller generators}
In this section we shall construct ray class invariants over imaginary quadratic fields whose minimal polynomials have relatively small coefficients.
\par

Let $K=\mathbb{Q}(\sqrt{-d})$ be an imaginary quadratic field with a square-free integer $d>0$,
$\F$ be a nontrivial proper integral ideal of $K$ and $\theta$ be as in (\ref{theta}). 
In what follows we adopt the notations of Section \ref{Generation of ray class fields by Siegel-Ramachandra invariants}.

\begin{lemma}\label{property of special value}
Let $\mathbf{r}, \mathbf{s}\in({1}/{N})\mathbb{Z}^2\setminus\mathbb{Z}^2$ for a positive integer $N\geq 2$. 
Then we achieve that
\begin{equation*}
\frac{g_{\mathbf{r}}(\theta)^{\gcd(2,N)\cdot N}}{g_{\mathbf{s}}(\theta)^{\gcd(2,N)\cdot N}}
\end{equation*}
lies in the ray class field $K_{(N)}=K_\F$ with $\F=N\O_K$.
\end{lemma}
\begin{proof}
It is immediate from Proposition \ref{complex multiplication} and \ref{quad relation}.
\end{proof}

\begin{lemma}\label{existence of ray class}
Let $\F=N\O_K=\prod_{i=1}^r \mathfrak{p}_i^{n_i}$ for an integer $N\geq 2$ and $p$ be an odd prime dividing $N$ \textup{(}if any\textup{)}.
Assume that $K_\F\neq K_{\F\P_i^{-n_i}}$  for every $i$.
Then there is an element $\beta\in\O_K$ prime to $6N$ for which $N_{K/\mathbb{Q}}(\beta)\equiv 1\pmod{\o_{K_\F}}$ and
the ray class $[(\beta)]\in\mathrm{Cl}(\F)$ is of order $k_p$, where
\begin{equation*}
k_p=\left\{
\begin{array}{ll} 
\displaystyle\frac{1}{\o_K}\left(p-\left(\frac{d_K}{p}\right)\right)& \textrm{if $p\nmid d_K$, $\mathrm{ord}_p(N)=1$ and $N=p$}\vspace{0.3cm}\\
\displaystyle\frac{1}{2}\left(p-\left(\frac{d_K}{p}\right)\right) & \textrm{if $p\nmid d_K$, $\mathrm{ord}_p(N)=1$ and $N\neq p$}\vspace{0.3cm}\\
p & \textrm{otherwise}.
\end{array}\right.
\end{equation*}
Here $(\frac{d_K}{p})$ is the Kronecker symbol.

\end{lemma}
\begin{proof}
Let 
\begin{equation*}
N'=\left\{
\begin{array}{ll}
4N & \textrm{if $3\,|\,N$}\vspace{0.2cm}\\
12N & \textrm{if $3\nmid N$}
\end{array}\right.
\end{equation*}
with prime decomposition $N'=\prod_\ell \ell^{n_\ell}$.
Then $n_p=\mathrm{ord}_p(N')=\mathrm{ord}_p(N)$ and  $\o_{K_\F}$ divides $N'$ (\cite[Chapter 9, Lemma 4.3]{Kubert}).
Hence it suffices to find $\beta\in\O_K$ for which $N_{K/\mathbb{Q}}(\beta)\equiv 1\pmod{N'}$ and 
the order of the ray class $[(\beta)]\in\mathrm{Cl}(\F)$ is $k_p$.
For simplicity, we let $m=p-\big(\frac{d_K}{p}\big)$ and define a homomorphism
\begin{equation*}
\begin{array}{rccc}
\widetilde{N_{K/\mathbb{Q},n}}:&(\O_K/n\O_K)^\times&\longrightarrow& (\mathbb{Z}/n\mathbb{Z})^\times\\
&\o+n\O_K&\longmapsto&N_{K/\mathbb{Q}}(\o)+n\mathbb{Z}
\end{array}
\end{equation*}
for each integer $n\geq 2$.
\par

\begin{itemize}
\item[\textbf{Case 1}.]  First, suppose that $p\nmid d_K$ and $\mathrm{ord}_p(N)=1$.
Note that the map 
\begin{equation*}
\widetilde{N_{K/\mathbb{Q},p}}:(\O_K/p\O_K)^\times\longrightarrow (\mathbb{Z}/p\mathbb{Z})^\times
\end{equation*}
is surjective and $\ker(\widetilde{N_{K/\mathbb{Q},p}})\cong \mathbb{Z}/m\mathbb{Z}$ (\cite[Corollary 5.2]{Koo}).
We choose $\o\in\O_K$ so that $\ker(\widetilde{N_{K/\mathbb{Q},p}})=\langle \o+p\O_K \rangle$.
Then there is $\o'\in\O_K$ for which for each prime $\ell$ dividing $N'$
\begin{equation*}
\o'\equiv \left\{
\begin{array}{ll}
\o\pmod{\ell^{n_\ell}\O_K}&\textrm{if $\ell=p$}\\
1\pmod{\ell^{n_\ell}\O_K}&\textrm{if $\ell\neq p$}
\end{array}\right.
\end{equation*}
by the Chinese remainder theorem.
We observe that $\o'$ is prime to $6N$, $\o'+N'\O_K$ is contained in $\ker(\widetilde{N_{K/\mathbb{Q},N'}})$ and it is of order $m$ in $(\O_K/N\O_K)^\times$.
If $N=p$, then the ray class $[(\o')]\in\mathrm{Cl}(\F)$ is of order $m/\o_K$ because 
\begin{equation*}
\{\varepsilon+p\O_K\in (\O_K/p\O_K)^\times~|~\varepsilon\in\O_K^\times\}\subset \ker(\widetilde{N_{K/\mathbb{Q},p}})=\langle \o'+p\O_K \rangle.
\end{equation*}
When $N\neq p$, we derive that
\begin{equation*}
\textrm{\big(the order of $[(\o')]$ in $\mathrm{Cl}(\F)$\big)}=\left\{
\begin{array}{ll}
{m}/{2} & \textrm{if $-1+N\O_K\in \langle \o'+N\O_K \rangle$}\vspace{0.2cm}\\
m & \textrm{otherwise}.
\end{array}\right. 
\end{equation*}
Therefore we can find $\beta\in\O_K$ with the desired properties.

\item[\textbf{Case 2}.] Now, assume that $\mathrm{ord}_p(N)>1$ or $p\,|\,d_K$.
Let
\begin{equation}\label{beta}
\beta=\left\{
\begin{array}{ll}
\displaystyle 1+\frac{2N}{p}\sqrt{-d}\phantom{\bigg(} &\textrm{if $p=3$}\vspace{0.3cm}\\
\displaystyle1+\frac{6N}{p}\sqrt{-d}\phantom{\bigg(} &\textrm{if $p\neq 3$}.
\end{array}\right.
\end{equation}
Then $N_{K/\mathbb{Q}}(\beta)\equiv 1\pmod{N'}$ and $\beta+N\O_K$ is of order $p$ in $(\O_K/N\O_K)^\times$ because  $p^2$ divides $Nd$.
And we claim that
\begin{equation*}
\{\varepsilon+N\O_K\in (\O_K/N\O_K)^\times~|~\varepsilon\in\O_K^\times\}\cap \langle\beta+N\O_K\rangle=\{1+N\O_K\}
\end{equation*}
since $p$ is an odd prime and $K_\F\neq K_{\F\P_i^{-n_i}}$  for every $i$.
Thus the ray class $[(\beta)]\in\mathrm{Cl}(\F)$ is of order $p$ as desired.
\end{itemize}
\end{proof}

\begin{remark}
If $p$ is an odd prime such that $p-\big(\frac{d_K}{p}\big)$ is a power of $2$, then $p$ is either a Mersenne prime or a Fermat prime.
Observe that 48 Mersenne primes and 5 Fermat primes are known as of May, 2014.
\end{remark}

\begin{theorem}\label{main theorem2}
Let $\F=\prod_{i=1}^r \mathfrak{p}_i^{n_i}$ be a nontrivial proper integral ideal of $K$ and $C'$ be an element of 
$\mathbf{G}~(\subset\mathrm{Cl}(\F))$ whose order is an odd prime $p$.
Assume that $K_\F\neq K_{\F\P_i^{-n_i}}$ and $|\mathbf{G}_i|>2$ for every $i$.
If $p>3$ or $|\mathbf{G}_i|>3$ for every $i$, then for any nonzero integer $n$ the unit
\begin{equation*}
\frac{g_\F(C')^n}{g_\F(C_0)^n}
\end{equation*}
generates $K_\F$ over $K$.
\par
Moreover, if $\F=N\O_K$ for an integer $N\geq 2$ and $C'=[(\beta')]$ for some $\beta'\in\O_K$ prime to $6N$ with $N_{K/\mathbb{Q}}(\beta')\equiv 1\pmod{\o_{K_\F}}$, then
\begin{equation*}
K_\F=K\left(\frac{g_{\left[\begin{smallmatrix}s/N\\t/N\end{smallmatrix}\right]}(\theta)^{m}}{g_{\left[\begin{smallmatrix}0\\ 1/N\end{smallmatrix}\right]}(\theta)^{m}}\right)
\end{equation*}
where $\beta'=s\theta+t$ with $s,t\in\mathbb{Z}$ and
\begin{equation*}
m=\left\{
\begin{array}{ll}
\gcd(N,3)&\textrm{if $N$ is odd},\vspace{0.2cm}\\
4\cdot\gcd(\frac{N}{2},3)&\textrm{if $N$ is even}.
\end{array}\right.
\end{equation*}
\end{theorem}

\begin{proof}
Let $F=K\big({g_\F(C')^n}/{g_\F(C_0)^n}\big)$.
Suppose $F\subsetneq K_\F$, namely,  $\mathrm{Cl}(K_\F/F)\neq\{1\}$.
We then claim that there exists a character $\chi$ of $\mathrm{Cl}(\F)$ such that $\chi|_{\mathrm{Cl}(K_\F/F)}\neq 1$, $\chi(C')\neq 1$ and $\mathfrak{p}_i\, |\, \mathfrak{f}_\chi$ for every $i$.
Indeed, one can achieve
\begin{equation*}
|\{ \textrm{characters $\chi\neq 1$ of $\mathrm{Cl}(\F)~|~ \chi(C')=1$}\}|=\frac{1}{p}[K_\F:K]-1.
\end{equation*}
And, it follows from (\ref{character1}) that
\begin{equation*}
|\left\{\textrm{characters $\chi$ of $\mathrm{Cl}(\F)~|~ \chi|_{\mathrm{Cl}(K_\F/F)}\neq 1$} \right\}|\geq \frac{1}{2}[K_\F:K]
>\frac{1}{p}[K_\F:K]-1.
\end{equation*}
Thus there is a character $\chi$ of $\mathrm{Cl}(\F)$ satisfying $\chi|_{\mathrm{Cl}(K_\F/F)}\neq 1$ and $\chi(C')\neq 1$.
Since $C'\in \mathbf{G}$, we may write $C'=[(\beta')]$ for some $\beta'\in\O_K$ which is prime to $\F$.
Choose $C''\in\mathrm{Cl}(K_\F/F)$ of prime order $\ell$ satisfying $\chi(C'')\neq 1$.
Here, we assume that $\F_\chi$ is not divided by $\P_i$ for some $i$.
We then consider the following two cases:
\begin{itemize}
\item[\textbf{Case 1}.]  First, consider the case $C''\not\in \mathbf{G}$.
The proof of Proposition \ref{existence of character} shows that there exists a nontrivial character $\psi$ of $\mathbf{G}_i$ such that $\psi([\beta'+{\mathfrak{p}_i^{n_i}}])\neq \chi(C')^{-1}$.
And, we can extend $\psi$ to a character $\psi'$ of $\mathbf{G}$ whose conductor is divisible only by $\mathfrak{p}_i$.
Since $C''\not\in \mathbf{G}$ and $\psi'((C'')^\ell)=1$, one can also extend $\psi'$ to a character $\psi''$ of $\mathrm{Cl}(\F)$ so as to have $\psi''(C'')=1$ by Lemma \ref{extension}.
Note that the character $\chi\psi''$ of $\mathrm{Cl}(\F)$ satisfies $\chi\psi''(C')\neq 1$, $\chi\psi''(C'')=\chi(C'')\neq 1$, $\P_i\,|\,\F_{\chi\psi''}$ and $\F_\chi\,|\,\F_{\chi\psi''}$.
And, we replace $\chi$ by $\chi\psi''$.

\item[\textbf{Case 2}.]  Next, suppose $C''\in \mathbf{G}$.
Let $\beta''\in\O_K$ such that $C''=[(\beta'')]$ in $\mathrm{Cl}(\F)$.
\begin{itemize}
\item[(i)] If $\mathbf{G}_i\neq\langle [\beta''+{\mathfrak{p}_i^{n_i}}]\rangle$, we can choose a trivial character $\psi_1$ of the proper subgroup $\langle [\beta''+{\mathfrak{p}_i^{n_i}}]\rangle$ of $\mathbf{G}_i$.
Since $\langle[\beta''+{\mathfrak{p}_i^{n_i}}]\rangle\cong\{1\}$ or $\mathbb{Z}/\ell\mathbb{Z}$ in the group $\mathbf{G}_i$, we have either $\langle[\beta'+{\mathfrak{p}_i^{n_i}}]\rangle=\langle[\beta''+{\mathfrak{p}_i^{n_i}}]\rangle\subsetneq \mathbf{G}_i$ or $\langle[\beta'+{\mathfrak{p}_i^{n_i}}]\rangle\cap\langle[\beta''+{\mathfrak{p}_i^{n_i}}]\rangle=\{1\}$.
And by using Lemma \ref{extension} we can extend $\psi_1$ to a nontrivial character $\psi$ of $\mathbf{G}_i$ such that 
$\psi(\beta'+{\mathfrak{p}_i^{n_i}})\neq\chi(C')^{-1}$ due to the fact $p\geq 3$.
\item[(ii)] Now assume that $\mathbf{G}_i=\langle [\beta''+{\mathfrak{p}_i^{n_i}}]\rangle\cong\mathbb{Z}/\ell\mathbb{Z}$.
Then $\ell>2$ because $|\mathbf{G}_i|>2$. 
If $\beta'+{\mathfrak{p}_i^{n_i}}\in\{\alpha+{\mathfrak{p}_i^{n_i}}\in (\O_K/\mathfrak{p}_i^{n_i})^\times~|~\alpha\in \O_K^\times \}$, there is a nontrivial character $\psi$ of $\mathbf{G}_i$ for which
$\psi([\beta''+{\P_i^{n_i}}])\neq \chi(C'')^{-1}$.
Oberve that $\psi([\beta'+{\P_i^{n_i}}])=1$.
\par
On the other hand, if $\beta'+{\mathfrak{p}_i^{n_i}}\not\in\{\alpha+{\mathfrak{p}_i^{n_i}}\in (\O_K/\mathfrak{p}_i^{n_i})^\times~|~\alpha\in \O_K^\times \}$, then $\langle [\beta'+{\mathfrak{p}_i^{n_i}}]\rangle=\langle [\beta''+{\mathfrak{p}_i^{n_i}}]\rangle$ and $p=\ell$.
And we see that $p,\ell>3$ by hypothesis.
Let $\psi_1$ be a character of $\mathbf{G}_i$ such that $\psi_1([\beta'+{\P_i^{n_i}}])=\zeta_p$.
Then $\psi, \psi^2,\ldots,\psi^{p-1}$ are distinct nontrivial characters of $\mathbf{G}_i$ and we derive
\begin{equation*}
\left|\{\psi_1^j~|~\psi_1^j([\beta'+{\P_i^{n_i}}])\neq \chi(C')^{-1} \}_{1\leq j\leq p-1}\right|=p-2.
\end{equation*}
Meanwhile, $\psi_1([\beta''+{\P_i^{n_i}}])=\zeta_p^s$ for some $1\leq s\leq p-1$.
If 
\begin{equation*}
\psi_1^{j_1}([\beta''+{\P_i^{n_i}}])=\psi_1^{j_2}([\beta''+{\P_i^{n_i}}])
\end{equation*}
for some $1\leq j_1,j_2\leq p-1$, then $\zeta_p^{s(j_1-j_2)}=1$ and so $j_1=j_2$.
Hence we have
\begin{equation*}
\left|\{\psi_1^j~|~\psi_1^j([\beta'+{\P_i^{n_i}}])\neq \chi(C')^{-1}, ~\psi_1^j([\beta''+{\P_i^{n_i}}])\neq \chi(C'')^{-1} \}_{1\leq j\leq p-1}\right|\geq p-3>0.
\end{equation*}
We choose a character $\psi$ of $\mathbf{G}_i$ in the above set.
\end{itemize}
Therefore, in any cases we can extend $\psi$ to a character $\psi''$ of $\mathrm{Cl}(\F)$ in a similar fashion,
 and it satisfies $\chi\psi''(C')\neq 1$, $\chi\psi''(C'')\neq 1$, $\P_i\,|\,\F_{\chi\psi''}$ and $\F_\chi\,|\,\F_{\chi\psi''}$.
Now, we replace $\chi$ by $\chi\psi''$.
\end{itemize}
By continuing this process for every $i$, we get the claim.
\par
Since $\chi$ is nontrivial and $\P_i\,|\,\F_\chi$ for every $i$, we have
$S_\F(\overline{\chi},g_\F)\neq 0$
by Proposition \ref{L-function relation} and Remark \ref{Stickremark}.
On the other hand, we deduce that
\begin{eqnarray*}
(\chi(C')-1)S_\F(\overline{\chi},g_\F)
&=&(\overline{\chi}(C'^{-1})-1)\sum_{C\in\mathrm{Cl}(\F)}\overline{\chi}(C)\log|g_\F(C)|\\
&=&\frac{1}{n}\sum_{C\in\mathrm{Cl}(\F)}\overline{\chi}(C)\log\left|\frac{g_\F(CC')^n}{g_\F(C)^n}\right|\\
&=&\frac{1}{n}\sum_{C\in\mathrm{Cl}(\F)}\overline{\chi}(C)\log\left|\left(\frac{g_\F(C')^n}{g_\F(C_0)^n}\right)^{\sigma(C)}\right| \quad\textrm{(by Proposition \ref{Galois action})}\\
&=&\frac{1}{n}\sum_{[C_1]\in\mathrm{Cl}(\F)/\mathrm{Cl}(K_\F/F)}
\left(\sum_{C_2\in\mathrm{Cl}(K_\F/F)}\overline{\chi}(C_1C_2)\log\left|\left(\frac{g_\F(C')^n}{g_\F(C_0)^n}\right)^{\sigma(C_1C_2)}\right|\right)\\
&=&\frac{1}{n}\sum_{[C_1]\in\mathrm{Cl}(\F)/\mathrm{Cl}(K_\F/F)}\overline{\chi}(C_1)\log\left|\left(\frac{g_\F(C')^n}{g_\F(C_0)^n}\right)^{\sigma(C_1)}\right|\left(\sum_{C_2\in\mathrm{Cl}(K_\F/F)}\overline{\chi}(C_2)\right)\\
&=&0,
\end{eqnarray*}
because $\overline{\chi}$ is nontrivial on $\mathrm{Cl}(K_\F/F)$.
And, it yields a contradiction since $\chi(C')-1\neq 0$.
Therefore, we conclude $F=K_\F$.
\par
If $\F=N\O_K$ for a positive integer $N\geq 2$ and $\beta'=s\theta+t$ is prime to $6N$ with $N_{K/\mathbb{Q}}(\beta')\equiv 1\pmod{\o_{K_\F}}$, then any $N$-th root of 
the value ${g_\F(C')}/{g_\F(C_0)}$ generates $K_\F$ over $K$ by Lemma \ref{Nth root}.
Write $\textrm{min}(\theta,\mathbb{Q})=X^2+B_\theta X+C_\theta$.
Then we have by definition $g_\F(C_0)=g_{\left[\begin{smallmatrix}0\\ 1/N\end{smallmatrix}\right]}(\theta)^{12N}$  and
\begin{equation*}
g_\F(C')=g_\F(C_0)^{\sigma(C')}=g_{
\scriptsize\left[\begin{matrix}   
t-B_\theta s&s\\
-C_\theta s&t
\end{matrix}\right]\left[\begin{matrix}0\\ 1/N\end{matrix}\right]}(\theta)^{12N}=g_{\left[\begin{smallmatrix}s/N\\ t/N\end{smallmatrix}\right]}(\theta)^{12N}
\end{equation*}
by Proposition \ref{principal class action}, \ref{Siegel property} and \ref{Galois action}.
Therefore, 
\begin{equation*}
K_\F=K\left(\frac{g_{\left[\begin{smallmatrix}s/N\\t/N\end{smallmatrix}\right]}(\theta)^{12}}{g_{\left[\begin{smallmatrix}0\\ 1/N\end{smallmatrix}\right]}(\theta)^{12}}\right),
\end{equation*}
and hence we get the conclusion by Lemma \ref{property of special value}.

\end{proof}
\begin{remark}
In like manner as in Lemma \ref{order condition} one can show that
$|\mathbf{G}_i|=3$ if and only if $\mathfrak{p}_i^{n_i}$ satisfies one of the following conditions:
\begin{itemize}
\item[] \textbf{Case 1}  : $K\neq \mathbb{Q}(\sqrt{-1}), \mathbb{Q}(\sqrt{-3})$ 
\begin{itemize}
\item[$\bullet$] $2$ is inert in $K$, $\P_i$ is lying over $2$ and $n_i=1$.
\item[$\bullet$] $3$ is not inert in $K$, $\P_i$ is lying over $3$ and $n_i=2$.
\item[$\bullet$] $7$ is not inert in $K$, $\P_i$ is lying over $7$ and $n_i=1$.
\end{itemize}
\item[] \textbf{Case 2}  : $K=\mathbb{Q}(\sqrt{-1})$ 
\begin{itemize}
\item[$\bullet$] $\P_i$ is lying over $13$ and $n_i=1$.
\end{itemize}
\item[] \textbf{Case 3} : $K=\mathbb{Q}(\sqrt{-3})$ 
\begin{itemize}
\item[$\bullet$] $\P_i$ is lying over $3$ and $n_i=3$.
\item[$\bullet$] $\P_i$ is lying over $19$ and $n_i=1$.
\end{itemize}
\end{itemize}

\end{remark}

\begin{remark}\label{additional case}
\begin{itemize}
\item[(i)] Suppose 
\begin{equation*}
\left(\frac{g_\F(C')}{g_\F(C_0)}\right)^{\sigma(C)}\neq \frac{g_\F(C')}{g_\F(C_0)}
\end{equation*}
for every $C\in \mathbf{G}\setminus\{1\}$.
Then we may  consider only the $\textbf{Case 1}$ in the proof of Theorem \ref{main theorem2}.
Thus one can prove that
Theorem \ref{main theorem2} is still valid for any prime $p$ with the assumptions $K_\F\neq K_{\F\P_i^{-n_i}}$ and $|\mathbf{G}_i|>2$ for every $i$. 
\item[(ii)]
Let $\F=N\O_K$ for a positive integer $N\geq 2$ and $\ell$ be an odd prime dividing $N$. 
For any odd prime $p$ dividing $k_\ell$, there exists $\beta'\in\O_K$ prime to $6N$ for which $N_{K/\mathbb{Q}}(\beta')\equiv 1\pmod{\o_{K_\F}}$ and the order of the ray class $[(\beta')]$ is $p$ by Lemma \ref{existence of ray class}.
\end{itemize}
\end{remark}

The following corollary would be an explicit example of Theorem \ref{main theorem2}.

\begin{corollary}\label{main corollary}
Let $\F=N\O_K$ for a positive integer $N\geq 2$ and $p$ be an odd prime dividing $N$ \textup{(}if any\textup{)}.
Further, we set $\beta\in\O_K$  as in (\ref{beta}) and 
$C'=[(\beta)]\in\mathrm{Cl}(\F)$.
Assume that $K_\F\neq K_{\F\P_i^{-n_i}}$, $|\mathbf{G}_i|>2$ for every $i$ and $p^2$ divides $Nd_K$.
\begin{itemize}
\item[(i)]
If $p=3$ and $|\mathbf{G}_i|>3$ for every $i$, then the special value
\begin{equation*}
\gamma=\left\{
\begin{array}{ll}
\displaystyle\frac{g_{\left[\begin{smallmatrix}2/3\\ 1/N\end{smallmatrix}\right]}(\theta)^3}{g_{\left[\begin{smallmatrix}0\\ 1/N\end{smallmatrix}\right]}(\theta)^3}&\textrm{if $d_K\equiv 0\pmod{4}$}\vspace{0.3cm}\\
\displaystyle\frac{g_{\left[\begin{smallmatrix}4/3\\ 2/3+1/N\end{smallmatrix}\right]}(\theta)^3}{g_{\left[\begin{smallmatrix}0\\ 1/N\end{smallmatrix}\right]}(\theta)^3}&\textrm{if $d_K\equiv 1\pmod{4}$}
\end{array}\right.
\end{equation*}
generates $K_\F ~(=K_{(N)})$ over $K$.
It is a unit in $K_\F$ and is a $4N$-th root of  ${g_\F(C')}/{g_\F(C_0)}$.
\item[(ii)]
If $p>3$, then the special value
\begin{equation*}
\gamma=\left\{
\begin{array}{ll}
\displaystyle\frac{g_{\left[\begin{smallmatrix}6/p\\ 1/N\end{smallmatrix}\right]}(\theta)}{g_{\left[\begin{smallmatrix}0\\ 1/N\end{smallmatrix}\right]}(\theta)}&\textrm{if $d_K\equiv 0\pmod{4}$}\vspace{0.3cm}\\
\displaystyle\frac{g_{\left[\begin{smallmatrix}12/p\\ 6/p+1/N\end{smallmatrix}\right]}(\theta)}{g_{\left[\begin{smallmatrix}0\\ 1/N\end{smallmatrix}\right]}(\theta)}&\textrm{if $d_K\equiv 1\pmod{4}$}
\end{array}\right.
\end{equation*}
generates $K_\F$ over $K$.
Further, it is a unit in $K_\F$ and is a $12N$-th root of ${g_\F(C')}/{g_\F(C_0)}$.
\end{itemize}
\end{corollary}

\begin{proof}
Note that $\beta$ is prime to $6N$, $N_{K/\mathbb{Q}}(\beta)\equiv 1\pmod{\o_{K_\F}}$ and the ray class $[(\beta)]\in\mathrm{Cl}(\F)$ is of order $p$ by Lemma \ref{existence of ray class}.
Hence if $p>3$ or $|\mathbf{G}_i|>3$ for every $i$, then we achieve by Theorem \ref{main theorem2}
\begin{equation*}
K_\F=K\left(\frac{g_\F(C')}{g_\F(C_0)}\right).
\end{equation*}
\begin{itemize}
\item[(i)]
If $p=3$, then we can write
\begin{equation*}
\beta=\left\{
\begin{array}{ll}
\displaystyle\frac{2N}{3}\theta+1\phantom{\bigg(} &\textrm{if $d_K\equiv 0\pmod{4}$}\vspace{0.2cm}\\
\displaystyle\frac{4N}{3}\theta+\frac{2N}{3}+1\phantom{\bigg(} &\textrm{if $d_K\equiv 1\pmod{4}$}.
\end{array}\right.
\end{equation*}
Here we observe that $g_\F(C_0)=g_{\left[\begin{smallmatrix}0\\ 1/N\end{smallmatrix}\right]}(\theta)^{12N}$ by definition and
\begin{equation*}
g_\F(C')=g_\F(C_0)^{\sigma(C')}=\left\{
\begin{array}{ll}
g_{\left[\begin{smallmatrix}2/3\\ 1/N\end{smallmatrix}\right]}(\theta)^{12N}&\textrm{if $d_K\equiv 0\pmod{4}$}\vspace{0.2cm}\\
g_{\left[\begin{smallmatrix}4/3\\2/3+1/N\end{smallmatrix}\right]}(\theta)^{12N} &\textrm{if $d_K\equiv 1\pmod{4}$}
\end{array}\right.
\end{equation*}
by Proposition \ref{principal class action}, \ref{Siegel property} and \ref{Galois action}.
Since $\zeta_N\in\mathcal{F}_N$, we see that the functions
\begin{equation*}
\frac{g_{\left[\begin{smallmatrix}2/3\\ 1/N\end{smallmatrix}\right]}(\tau)^3}{g_{\left[\begin{smallmatrix}0\\ 1/N\end{smallmatrix}\right]}(\tau)^3}\quad\textrm{and}\quad
\frac{g_{\left[\begin{smallmatrix}4/3\\ 2/3+1/N\end{smallmatrix}\right]}(\tau)^3}{g_{\left[\begin{smallmatrix}0\\ 1/N\end{smallmatrix}\right]}(\tau)^3}
\end{equation*}
belong to $\mathcal{F}_N$ by Proposition \ref{quad relation}.
Thus its special value at $\theta$ lies in $K_\F=K_{(N)}$ by Proposition \ref{complex multiplication},
and so it generates $K_\F$ over $K$.
\item[(ii)]
If $p>3$, then we may write
\begin{equation*}
\beta=\left\{
\begin{array}{ll}
\displaystyle\frac{6N}{p}\theta+1\phantom{\bigg(} &\textrm{if $d_K\equiv 0\pmod{4}$}\vspace{0.2cm}\\
\displaystyle\frac{12N}{p}\theta+\frac{6N}{p}+1\phantom{\bigg(} &\textrm{if $d_K\equiv 1\pmod{4}$}.
\end{array}\right.
\end{equation*}
And, in a similar way one can readily show that the special value $\gamma^p$
generates $K_\F$ over $K$.
On the other hand, $\gamma^{12}$ is an element of $K_\F$ by Theorem \ref{main theorem2}.
Therefore, we get
\begin{equation*}
K_\F=K(\gamma)
\end{equation*}
owing to the fact $\gcd(p,12)=1$.
\end{itemize}
\end{proof}

\begin{example}\label{first example}
Let $K=\mathbb{Q}(\sqrt{-7})$ and $\theta=({-1+\sqrt{-7}})/{2}$.  
Then the class number $h_K$ is 1.
\begin{itemize}

\item[(i)]
Let $\F=9\O_K$.
Then 3 is inert in $K$ and so $|\mathbf{G}_i|>3$ for every $i$.
Hence the special value
\begin{equation*}
\gamma=\frac{g_{\left[\begin{smallmatrix}4/3\\ 7/9\end{smallmatrix}\right]}(\theta)^3}{g_{\left[\begin{smallmatrix}0\\ 1/9\end{smallmatrix}\right]}(\theta)^3}
\end{equation*}
generates $K_{(9)}$ over $K$ by Corollary \ref{main corollary}.
\par
Here we note that $\mathrm{Gal}(K_{(9)}/K)\cong W_{9,\theta}/\{\pm I_2\}$ by Proposition \ref{shimura reciprocity law}, and hence we obtain
\begin{equation*}
\begin{array}{rl}
W_{9,\theta}/\{\pm I_2\}=&\bigg\{ 
\footnotesize\left[\begin{matrix}
1&0\\
0&1
\end{matrix}\right],
\left[
\begin{matrix}
2&0\\
0&2
\end{matrix}\right],
\left[
\begin{matrix}
4&0\\
0&4
\end{matrix}\right],
\left[
\begin{matrix}
8&7\\
1&0
\end{matrix}\right],
\left[
\begin{matrix}
0&7\\
1&1
\end{matrix}\right],
\left[
\begin{matrix}
1&7\\
1&2
\end{matrix}\right],
\left[
\begin{matrix}
2&7\\
1&3
\end{matrix}\right],
\left[
\begin{matrix}
3&7\\
1&4
\end{matrix}\right],\vspace{0.2cm}\\
&
\footnotesize\left[\begin{matrix}
4&7\\
1&5
\end{matrix}\right],
\left[
\begin{matrix}
5&7\\
1&6
\end{matrix}\right],
\left[
\begin{matrix}
6&7\\
1&7
\end{matrix}\right],
\left[
\begin{matrix}
7&7\\
1&8
\end{matrix}\right],
\left[
\begin{matrix}
7&5\\
2&0
\end{matrix}\right],
\left[
\begin{matrix}
8&5\\
2&1
\end{matrix}\right],
\left[
\begin{matrix}
0&5\\
2&2
\end{matrix}\right],
\left[
\begin{matrix}
1&5\\
2&3
\end{matrix}\right],\vspace{0.2cm}\\
&
\footnotesize\left[\begin{matrix}
2&5\\
2&4
\end{matrix}\right],
\left[
\begin{matrix}
3&5\\
2&5
\end{matrix}\right],
\left[
\begin{matrix}
4&5\\
2&6
\end{matrix}\right],
\left[
\begin{matrix}
5&5\\
2&7
\end{matrix}\right],
\left[
\begin{matrix}
6&5\\
2&8
\end{matrix}\right],
\left[
\begin{matrix}
7&3\\
3&1
\end{matrix}\right],
\left[
\begin{matrix}
8&3\\
3&2
\end{matrix}\right],
\left[
\begin{matrix}
1&3\\
3&4
\end{matrix}\right],\vspace{0.2cm}\\
&
\footnotesize\left[\begin{matrix}
2&3\\
3&5
\end{matrix}\right],
\left[
\begin{matrix}
4&3\\
3&7
\end{matrix}\right],
\left[
\begin{matrix}
5&3\\
3&8
\end{matrix}\right],
\left[
\begin{matrix}
5&1\\
4&0
\end{matrix}\right],
\left[
\begin{matrix}
6&1\\
4&1
\end{matrix}\right],
\left[
\begin{matrix}
7&1\\
4&2
\end{matrix}\right],
\left[
\begin{matrix}
8&1\\
4&3
\end{matrix}\right],
\left[
\begin{matrix}
0&1\\
4&4
\end{matrix}\right],\vspace{0.2cm}\\
&
\footnotesize\left[\begin{matrix}
1&1\\
4&5
\end{matrix}\right],
\left[
\begin{matrix}
2&1\\
4&6
\end{matrix}\right],
\left[
\begin{matrix}
3&1\\
4&7
\end{matrix}\right],
\left[
\begin{matrix}
4&1\\
4&8
\end{matrix}\right]\bigg\}.
\end{array}
\end{equation*}

And, in general, if $\alpha \in \mathrm{GL}_2(\mathbb{Z}/N\mathbb{Z})/\{\pm I_2\}$ for a positive integer $N\geq 2$, then 
one can find $\alpha'\in \mathrm{SL}_2(\mathbb{Z})$ satisfying
\begin{equation*}
\alpha\equiv\left[
\begin{matrix}
1&0\\
0&\det(\alpha)
\end{matrix}\right]
\cdot\alpha'\pmod{N}
\end{equation*}
by (\ref{decomposition}).
If a function ${g_{\mathbf{r}}(\tau)^m}/{g_{\mathbf{s}}(\tau)^m}$ lies in $\mathcal{F}_N$ for some $\mathbf{r},\mathbf{s}\in\mathbb{Q}^2\setminus \mathbb{Z}^2$ and $m\in\mathbb{Z}_{>0}$, we attain
\begin{equation*}
\left(\frac{g_{\mathbf{r}}(\tau)^m}{g_{\mathbf{s}}(\tau)^m} \right)^\alpha
=\frac{a(\mathbf{r})}{a(\mathbf{s})}\cdot
\frac{g_{ {^t}\alpha'\left[
\begin{smallmatrix}
1&0\\
0&\det(\alpha)
\end{smallmatrix}\right]\mathbf{r}}(\tau)^m}{g_{ {^t}\alpha'\left[
\begin{smallmatrix}
1&0\\
0&\det(\alpha)
\end{smallmatrix}\right]\mathbf{s}}(\tau)^m} 
\end{equation*}
by (\ref{diagonal action}) and Proposition \ref{Siegel property}.
Here
\begin{equation*}
a(\mathbf{r})=\left\{
\begin{array}{ll}
-1&\textrm{if $m N r_2(r_1-1)$ is odd and $\det(\alpha)$ is even}\vspace{0.2cm}\\
1&\textrm{otherwise}
\end{array}\right.
\end{equation*}
for $\mathbf{r}=\left[
\begin{matrix}
r_1\\
r_2
\end{matrix}\right]\in\mathbb{Q}^2\setminus \mathbb{Z}^2$.
In our case, the function ${g_{\left[\begin{smallmatrix}4/3\\ 7/9\end{smallmatrix}\right]}(\tau)^3}/{g_{\left[\begin{smallmatrix}0\\ 1/9\end{smallmatrix}\right]}(\tau)^3}$ belongs to $\mathcal{F}_9$ and so
 we can find all conjugates of $\gamma$ over $K$ by using Proposition \ref{shimura reciprocity law}.
Therefore, we derive the minimal polynomial of $\gamma$ over $\mathbb{Q}$ as
\begin{equation*}
{\footnotesize
\begin{array}{ccl}
\mathrm{min}(\gamma,\mathbb{Q})&=&\displaystyle\prod_{\tau\in\mathrm{Gal}(K_{(9)}/K)}(X-\gamma^\tau) (X-\overline{\gamma^\tau})\vspace{0.1cm}\\
&=&X^{72}+90X^{71}
+1152X^{70}
-22371X^{69}
+458820X^{68}
-29836953X^{67}\vspace{0.1cm}\\
&&
+491027613X^{66}
-1938660903X^{65}
-20725828920X^{64}
+218606201947X^{63}\vspace{0.1cm}\\
&&
-87981391440X^{62}
-9726726330846X^{61}
+74685511048146X^{60}
\vspace{0.1cm}\\
&&
-296777453271966X^{59}
+741369035579850X^{58}
-1250575046567529X^{57}
\vspace{0.1cm}\\
&&
+1668303706335570X^{56}
-3404755297594260X^{55}
+12286071601634287X^{54}
\vspace{0.1cm}\\
&&
-32591232085278402X^{53}
+35114715622084023X^{52}
+37809379416794814X^{51}
\vspace{0.1cm}\\
&&
-111424993786127475X^{50}
-44163687277340892X^{49}
+282536182740148884X^{48}
\vspace{0.1cm}\\
&&
-43713385246904949X^{47}
-422588747471994153X^{46}
+222731731243593448X^{45}\vspace{0.1cm}\\
&&
+334105708870044999X^{44}
-414268957496144781X^{43}
+13834474218095754X^{42}\vspace{0.1cm}\\
&&
+634423686065669232X^{41}
-404320599974193246X^{40}
-761298152585541393X^{39}
\vspace{0.1cm}\\
&&
+489778367476257828X^{38}
+416185685059783914X^{37}
-442068360347754785X^{36}
\vspace{0.1cm}\\
&&
+416185685059783914X^{35}
+489778367476257828X^{34}
-761298152585541393X^{33}
\vspace{0.1cm}\\
&&
-404320599974193246X^{32}
+634423686065669232X^{31}
+13834474218095754X^{30}
\vspace{0.1cm}\\
&&
-414268957496144781X^{29}
+334105708870044999X^{28}
+222731731243593448X^{27}
\vspace{0.1cm}\\
&&
-422588747471994153X^{26}
-43713385246904949X^{25} 
+282536182740148884X^{24}
\vspace{0.1cm}\\
&&
-44163687277340892X^{23}
-111424993786127475X^{22}
+37809379416794814X^{21}
\vspace{0.1cm}\\
&&
+35114715622084023X^{20} 
-32591232085278402X^{19}
+12286071601634287X^{18}
\vspace{0.1cm}\\
&& 
-3404755297594260X^{17}
+1668303706335570X^{16} 
-1250575046567529X^{15}
\vspace{0.1cm}\\
&&
+741369035579850X^{14}
-296777453271966X^{13}
+74685511048146X^{12}
\vspace{0.1cm}\\
&&
-9726726330846X^{11}
-87981391440X^{10}
+218606201947X^9
-20725828920X^8
\vspace{0.1cm}\\
&&
-1938660903X^7
+491027613X^6
-29836953X^5
+458820X^4
-22371X^3\vspace{0.1cm}\\
&&
+1152X^2
+90X+1,
\end{array}}
\end{equation*}
which claims that $\gamma$ is a unit as desired.
\par
On the other hand, the Siegel-Ramachandra invariant
\begin{equation*}
g_\F(C_0)=g_{\left[\begin{smallmatrix}0\\ 1/9\end{smallmatrix}\right]}(\theta)^{108}
\end{equation*}
also generates $K_{(9)}$ over $K$ by Theorem \ref{main theorem}.
Observe that it is a real algebraic integer, and so its minimal polynomial over $K$ has integer coefficients (\cite[Theorem 3.5 and Remark 3.6]{Jung}). 
One can then compute the minimal polynomial of $g_{\left[\begin{smallmatrix}0\\ 1/9\end{smallmatrix}\right]}(\theta)^{108}$ over $K$ as follows:
\begin{equation*}
{\footnotesize
\begin{array}{ccl}
\mathrm{min}\left(g_{\left[\begin{smallmatrix}0\\ 1/9\end{smallmatrix}\right]}(\theta)^{108},K\right)&\approx& 
X^{36}-5.8014\times 10^{16}X^{35}+1.2510\times 10^{33}X^{34}-1.2073\times 10^{49}X^{33}\\
&&+5.2876\times 10^{64}X^{32}-1.3770\times 10^{80}X^{31}+4.5041\times 10^{95}X^{30}\vspace{0.15cm}\\
&&+7.1821\times 10^{109}X^{29}+3.5929\times 10^{125}X^{28}+6.0405\times 10^{140}X^{27}\vspace{0.15cm}\\
&&-2.6727\times 10^{153}X^{26}+4.0906\times 10^{166}X^{25}+1.5461\times 10^{178}X^{24}\vspace{0.15cm}\\
&&+2.5470\times 10^{189}X^{23}-8.8165\times 10^{197}X^{22}+1.2086\times 10^{206}X^{21}\vspace{0.15cm}\\
&&-6.5232\times 10^{213}X^{20}+1.1931\times 10^{221}X^{19}+1.1532\times 10^{226}X^{18}\vspace{0.15cm}\\
&&+1.8902\times 10^{231}X^{17}+5.0656\times 10^{233}X^{16}+4.0609\times 10^{234}X^{15}\vspace{0.15cm}\\
&&+1.3087\times 10^{235}X^{14}+1.8279\times 10^{235}X^{13}+1.0208\times 10^{235}X^{12}\vspace{0.15cm}\\
&&+1.3732\times 10^{234}X^{11}-5.1693\times 10^{229}X^{10}+1.5848\times 10^{225}X^{9}\vspace{0.15cm}\\
&&+1.2122\times 10^{218}X^{8}-1.7829\times 10^{211}X^{7}+1.2402\times 10^{204}X^{6}\vspace{0.15cm}\\
&&+5.8968\times 10^{184}X^{5}+7.7183\times 10^{164}X^{4}+1.3109\times 10^{144}X^{3}\vspace{0.15cm}\\
&&-1.2605\times 10^{110}X^{2}+1.1125\times 10^{76}X+5.8150\times 10^{25}.
\end{array}}
\end{equation*}
But, we notice here that the coefficients of $\mathrm{min}(\gamma,\mathbb{Q})$ are much smaller than those of $\mathrm{min}\big(g_{\left[\begin{smallmatrix}0\\ 1/9\end{smallmatrix}\right]}(\theta)^{108},K\big)$.

\item[(ii)]
Let $\F=5\O_K$. 
Then 5 is inert in $K$ and so $|\mathbf{G}_i|>3$ for every $i$.
By Lemma \ref{existence of ray class} we get an element $\beta\in\O_K$ prime to $30$ for which $N_{K/\mathbb{Q}}(\beta)\equiv 1\pmod{\o_{K_\F}}$ and the ray class $[(\beta)]\in\mathrm{Cl}(\F)$ is of order $3$.
Indeed, $\beta=6\sqrt{-5}+7$ satisfies these conditions.
Since $\beta=12\theta+13$, the special value
\begin{equation*}
\gamma=\frac{g_{\left[\begin{smallmatrix}12/5\\ 13/5\end{smallmatrix}\right]}(\theta)}{g_{\left[\begin{smallmatrix}0\\ 1/5\end{smallmatrix}\right]}(\theta)}
\end{equation*}
generates $K_{(5)}$ over $K$ by Theorem \ref{main theorem2}.
\par
Since the function ${g_{\left[\begin{smallmatrix}12/5\\ 13/5\end{smallmatrix}\right]}(\tau)}/{g_{\left[\begin{smallmatrix}0\\ 1/5\end{smallmatrix}\right]}(\tau)}$ belongs to $\mathcal{F}_{25}$ by Proposition \ref{complex multiplication}, 
in order to estimate the minimal polynomial of $\gamma$ over $\mathbb{Q}$ we need to describe the action of $\mathrm{Gal}(K_{(25)}/K)$.
Since
\begin{equation*}
[K_{(25)}:K]=300 \quad\textrm{and}\quad [K_{(5)}:K]=12,
\end{equation*}
we have
\begin{equation*}
\prod_{\tau\in\mathrm{Gal}(K_{(25)}/K)}(X-\gamma^\tau)=\mathrm{min}(\gamma,K)^{25}
\end{equation*}
and hence we can find all conjugates of $\gamma$ over $K$ in a similar way as in (i).
And, the minimal polynomial of $\gamma$ over $\mathbb{Q}$ is
\begin{equation*}
{\footnotesize
\begin{array}{ccl}
\mathrm{min}(\gamma,\mathbb{Q})&=& 
X^{24}
-3X^{23}
+3X^{22}
-3X^{21}
+11X^{20}
-3X^{19}
+24X^{18}
-24X^{17}
+4X^{16}-18X^{15}\vspace{0.1cm}\\
&&
+53X^{14}
-39X^{13}
-11X^{12}
-39X^{11}
+53X^{10}
-18X^{9}
+4X^{8}
-24X^{7}
+24X^{6}
\vspace{0.1cm}\\
&&-3X^{5}+11X^{4}-3X^{3}+3X^{2}-3X+1.
\end{array}}
\end{equation*}

On the other hand, the Siegel-Ramachandra invariant
\begin{equation*}
g_\F(C_0)=g_{\left[\begin{smallmatrix}0\\ 1/5\end{smallmatrix}\right]}(\theta)^{60}
\end{equation*}
also generates $K_{(5)}$ over $K$ by Theorem \ref{main theorem} and its minimal polynomial over $K$ is
\begin{equation*}
{\footnotesize
\begin{array}{ccl}
\mathrm{min}\left(g_{\left[\begin{smallmatrix}0\\ 1/5\end{smallmatrix}\right]}(\theta)^{60},K\right)&=& 
X^{12}
-531770250 X^{11}
+52496782397690625 X^{10}
\\
&&
+12347712418332056278906250 X^{9}\vspace{0.15cm}\\
&&
+517064715767117085870064453125000 X^{8}\vspace{0.15cm}\\
&&
+5105793070560695709489861859357910156250 X^{7}\vspace{0.15cm}\\
&&
+30043009324891990472511274397078094482421875X^{6}\vspace{0.15cm}\\
&&
+356967020673816044809943223760162353515625000X^{5}\vspace{0.15cm}\\
&&
+5338772150500577473141088454029560089111328125X^{4}\vspace{0.15cm}\\
&&
+263440400470778826352188828480243682861328125X^{3}\vspace{0.15cm}\\
&&
-4471591562072879160572290420532226562500 X^{2}\vspace{0.15cm}\\
&&
+62983472112150751054286956787109375X
+931322574615478515625.
\end{array}}
\end{equation*}

\end{itemize}
\end{example}

\begin{example}\label{even example}
Let $K=\mathbb{Q}(\sqrt{-5})$ and $\theta=\sqrt{-5}$. 
Then the class number $h_K$ is 2.
\begin{itemize}

\item[(i)]
Let $\F=4\O_K$.
Then $2$ is ramified in $K$ and hence $|\mathbf{G}_i|>2$ for every $i$.
Since $\mathbf{G}\cong W_{4,\theta}/\{\pm I_2\}$ and
\begin{equation*}
W_{4,\theta}/\{\pm I_2\}=\left\{ 
\footnotesize\left[\begin{matrix}
1&0\\
0&1
\end{matrix}\right],
\left[
\begin{matrix}
1&2\\
2&1
\end{matrix}\right],
\left[
\begin{matrix}
0&3\\
1&0
\end{matrix}\right],
\left[
\begin{matrix}
2&3\\
1&2
\end{matrix}\right]\right\},
\end{equation*}
one can check by using Proposition \ref{principal class action}, \ref{Siegel property} that
the ray class $C'=[(3\sqrt{-5}+2)]$ in $\mathrm{Cl}(\F)$ is of order $2$ and satisfies
\begin{equation*}
\left(\frac{g_\F(C')}{g_\F(C_0)}\right)^{\sigma(C)}\neq \frac{g_\F(C')}{g_\F(C_0)}
\end{equation*}
for every $C\in \mathbf{G}\setminus\{1\}$.
Thus we see from Remark \ref{additional case} that
\begin{equation*}
K_{(4)}=K\left(\frac{g_\F(C')}{g_\F(C_0)}\right).
\end{equation*}
Furthermore, since $N_{K/\mathbb{Q}}(3\sqrt{-5}+2)\equiv 1\pmod{48}$,
the special value
\begin{equation*}
\gamma=\frac{g_{\left[\begin{smallmatrix}3/4\\2/4\end{smallmatrix}\right]}(\theta)^{4}}{g_{\left[\begin{smallmatrix}0\\ 1/4\end{smallmatrix}\right]}(\theta)^{4}}
\end{equation*}
also generates $K_{(4)}$ over $K$.
\par
Now that the function ${g_{\left[\begin{smallmatrix}3/4\\ 2/4\end{smallmatrix}\right]}(\tau)^4}/{g_{\left[\begin{smallmatrix}0\\ 1/4\end{smallmatrix}\right]}(\tau)^4}$ lies in $\mathcal{F}_{8}$ by Proposition \ref{complex multiplication}, in order to estimate the minimal polynomial of $\gamma$ over $\mathbb{Q}$ we should know the action of $\mathrm{Gal}(K_{(8)}/K)$.
The form class group $C(d_K)$ of discriminant $d_K=-20$ consists of two reduced quadratic forms
\begin{equation*}
Q_1=[1,0,5]\quad\textrm{and}\quad Q_2=[2,2,3].
\end{equation*}
Thus we have
\begin{equation*}
\theta_{Q_1}=\sqrt{-5},~\beta_{Q_1}=
\left[\begin{matrix}
1&0\\
0&1
\end{matrix}\right]\quad\textrm{and}\quad
\theta_{Q_2}=\frac{-1+\sqrt{-5}}{2},~\beta_{Q_2}=
\left[\begin{matrix}
-1&-3\\
1&0
\end{matrix}\right].
\end{equation*}
Then $W_{8,\theta}/\{\pm I_2\}$ and $C(d_K)$ determine the group $\mathrm{Gal}(K_{(8)}/K)$ by Proposition \ref{shimura reciprocity law}.
And, it follows from Proposition \ref{order of ray class field} that
\begin{equation*}
[K_{(8)}:K]=32 \quad\textrm{and}\quad [K_{(4)}:K]=8,
\end{equation*}
and so we attain
\begin{equation*}
\prod_{\tau\in\mathrm{Gal}(K_{(8)}/K)}(X-\gamma^\tau)=\mathrm{min}(\gamma,K)^{4}.
\end{equation*}
Therefore, the minimal polynomial of $\gamma$ over $\mathbb{Q}$ is
{\footnotesize\begin{equation*}
\mathrm{min}(\gamma,\mathbb{Q})= X^{16}
+3024X^{14}
+128700X^{12}
+53296X^{10}
-124026X^{8}
+53296X^{6}
+128700X^{4}
+3024X^{2}+1.
\end{equation*}}
On the other hand, we deduce
\begin{equation*}
{\footnotesize
\begin{array}{ccl}
\mathrm{min}\left(g_{\left[\begin{smallmatrix}0\\ 1/4\end{smallmatrix}\right]}(\theta)^{48},K\right)&=& 
X^{8}
-1597237832768X^{7}
-15846881298723072X^{6}
\\
&&
-26992839895872106496X^{5}
+655492492138238044037120X^{4}\vspace{0.15cm}
\\
&&
-169817799503383057556832256X^{3}
-20680171763956163581837312X^{2}\vspace{0.15cm}
\\
&&
-2550974942361763927031808X
+16777216.
\end{array}}
\end{equation*}

\item[(ii)]
Let $\F=5\O_K$.
Then $5$ is ramified in $K$ and hence $|\mathbf{G}_i|>2$ for every $i$.
Since $5$ divides $d_K=-20$, the special value
\begin{equation*}
\gamma=\frac{g_{\left[\begin{smallmatrix}6/5\\ 1/5\end{smallmatrix}\right]}(\theta)}{g_{\left[\begin{smallmatrix}0\\ 1/5\end{smallmatrix}\right]}(\theta)}
\end{equation*}
generates $K_{(5)}$ over $K$ by Corollary \ref{main corollary}.
\par
It then follows from Proposition \ref{complex multiplication}, \ref{order of ray class field} that  ${g_{\left[\begin{smallmatrix}6/5\\ 1/5\end{smallmatrix}\right]}(\tau)}/{g_{\left[\begin{smallmatrix}0\\ 1/5\end{smallmatrix}\right]}(\tau)}\in\mathcal{F}_{25}$ and 
\begin{equation*}
[K_{(25)}:K]=500 \quad\textrm{and}\quad [K_{(5)}:K]=20.
\end{equation*}
And, we induce
\begin{equation*}
\prod_{\tau\in\mathrm{Gal}(K_{(25)}/K)}(X-\gamma^\tau)=\mathrm{min}(\gamma,K)^{25}.
\end{equation*}
Observe that
\begin{equation*}
\theta_{Q_1}=\sqrt{-5},~\beta_{Q_1}=
\left[\begin{matrix}
1&0\\
0&1
\end{matrix}\right]\quad\textrm{and}\quad
\theta_{Q_2}=\frac{-1+\sqrt{-5}}{2},~\beta_{Q_2}=
\left[\begin{matrix}
2&1\\
0&1
\end{matrix}\right]
\end{equation*}
in this case.
Therefore, we obtain the minimal polynomial of $\gamma$ over $\mathbb{Q}$ as follows:
\begin{equation*}
{\footnotesize
\begin{array}{ccl}
\mathrm{min}(\gamma,\mathbb{Q})&=& 
X^{40}
+10X^{39}
+50X^{38}
+170X^{37}
+420X^{36}
+732X^{35}
+965X^{34}
+1380X^{33}\vspace{0.1cm}\\
&&
+2545X^{32}
+4460X^{31}
+6798X^{30}
+7880X^{29}
+1605X^{28}
-11800X^{27}
-11035X^{26}\vspace{0.1cm}\\
&&
+15554X^{25}
+31975X^{24}
+3050X^{23}
-29125X^{22}
-20050X^{21}
-2145X^{20}
-20050X^{19}\vspace{0.1cm}\\
&&
-29125X^{18}
+3050X^{17}
+31975X^{16}
+15554X^{15}
-11035X^{14}
-11800X^{13}
+1605X^{12}\vspace{0.1cm}\\
&&
+7880X^{11}
+6798X^{10}
+4460X^{9}
+2545X^{8}
+1380X^{7}
+965X^{6}
+732X^{5}
+420X^{4}\vspace{0.1cm}\\
&&
+170X^{3}
+50X^{2}
10X+1.
\end{array}}
\end{equation*}

\end{itemize}

\end{example}

\section{Application to quadratic Diophantine equations}

Let $n$ be a square-free positive integer, $K=\mathbb{Q}(\sqrt{-n})$ and $\theta$ be as in (\ref{theta}).
We assume $-n\equiv 2,3\pmod{4}$ so that $d_K\equiv 0\pmod{4}$ and $\O_K=\mathbb{Z}[\sqrt{-n}]$.
By means of ray class invariants over $K$, Cho \cite{Cho} provided a criterion whether a given odd prime $p$ can be written in the form $p=x^2+ny^2$ for some $x,y\in\mathbb{Z}$ with additional conditions $x\equiv 1 ~(\bmod{~N}),~ y\equiv 0~(\bmod{~N})$ for each positive integer $N$.

\begin{proposition}\label{Diophantine}
For a positive integer $N$, we let $f_{N}(X)\in\mathbb{Z}[X]$ be the minimal polynomial of a real algebraic integer which generates $K_{(N)}$ over $K$.
If an odd prime $p$ divides neither $nN$ nor the discriminant of $f_{N}(X)$, then
\begin{equation*}
\left(
\begin{array}{c}
\textrm{$p=x^2+ny^2$ with $x,y\in\mathbb{Z}$}\\
x\equiv 1 ~(\bmod{~N}),~ y\equiv 0~(\bmod{~N})
\end{array}\right)\Longleftrightarrow
\left(
\begin{array}{c}
\textrm{$\left(\frac{-n}{p}\right)=1$ and $f_{N}(X)\equiv 0~(\bmod{~p})$}\\
\textrm{has an integer solution}
\end{array}\right)
\end{equation*}
where $(\frac{-n}{p})$ is the Kronecker symbol.

\end{proposition}
\begin{proof}
\cite[Theorem 1]{Cho}.
\end{proof}

\begin{lemma}\label{real generator}
Let $N\geq 2$ be an integer.
\begin{itemize}
\item[\textup{(i)}] For $s\in \mathbb{Z}\setminus N\mathbb{Z}$,
\begin{equation*}
\frac{g_{\left[\begin{smallmatrix}0\\ s/N\end{smallmatrix}\right]}(\theta)}{g_{\left[\begin{smallmatrix}0\\ 1/N\end{smallmatrix}\right]}(\theta)}
\end{equation*}
is a real number.
\item[\textup{(ii)}]
For $t\in \mathbb{Z}$,
\begin{equation*}
ie^{ \frac{ t }{2N}\pi i} \frac{g_{\left[\begin{smallmatrix}1/2\\ t/N\end{smallmatrix}\right]}(\theta)}{g_{\left[\begin{smallmatrix}0\\ 1/N\end{smallmatrix}\right]}(\theta)}
\end{equation*}
is a real number.
\end{itemize}
\end{lemma}

\begin{proof}
We obtain by definition
\begin{equation*}
\begin{array}{ccl}
e^{ \frac{t }{2N}\pi i} g_{\left[\begin{smallmatrix}1/2\\ t/N\end{smallmatrix}\right]}(\theta)&=&
-q_{\theta}^{\frac{1}{2}\mathbf{B}_2(\frac{1}{2})}(1-q_\theta^{\frac{1}{2}}\zeta_N^t)\displaystyle \prod_{n=1}^\infty(1-q_\theta^{n+\frac{1}{2}}\zeta_N^t)
(1-q_\theta^{n-\frac{1}{2}}\zeta_N^{-t})\\
&=&-q_{\theta}^{\frac{1}{2}\mathbf{B}_2(\frac{1}{2})}\displaystyle \prod_{n=0}^\infty(1-q_\theta^{n+\frac{1}{2}}\zeta_N^t)
\prod_{n=1}^\infty(1-q_\theta^{n-\frac{1}{2}}\zeta_N^{-t})\\
&=&-q_{\theta}^{\frac{1}{2}\mathbf{B}_2(\frac{1}{2})}\displaystyle \prod_{n=1}^\infty(1-q_\theta^{n-\frac{1}{2}}\zeta_N^t)
(1-q_\theta^{n-\frac{1}{2}}\zeta_N^{-t})\\
&=&-q_{\theta}^{\frac{1}{2}\mathbf{B}_2(\frac{1}{2})}\displaystyle \prod_{n=1}^\infty
\left\{1-q_\theta^{n-\frac{1}{2}}(\zeta_N^t+\zeta_N^{-t})+q_\theta^{2n-1}\right\}.
\end{array}
\end{equation*}
And, we deduce
\begin{equation*}
\begin{array}{ccl}
-i g_{\left[\begin{smallmatrix}0\\ s/N\end{smallmatrix}\right]}(\theta)
&=&iq_{\theta}^{\frac{1}{12}}e^{-\frac{s}{N}\pi i }(1-\zeta_N^s)\displaystyle \prod_{n=1}^\infty(1-q_\theta^{n}\zeta_N^s)
(1-q_\theta^{n}\zeta_N^{-s})\\
&=&iq_{\theta}^{\frac{1}{12}}(\zeta_{2N}^{-s}-\zeta_{2N}^s)\displaystyle \prod_{n=1}^\infty\left\{1-q_\theta^{n}(\zeta_N^s+\zeta_N^{-s})+q_\theta^{2n}\right\}.
\end{array}
\end{equation*}
Since $q_\theta$, $\zeta_N^t+\zeta_N^{-t}$ and $i(\zeta_{2N}^{-s}-\zeta_{2N}^s)$ are real numbers, we get the lemma.

\end{proof}

\begin{theorem}\label{real ray class invariant}
Let $N$ be a positive integer and $\F=N\O_K$ with prime ideal factorization
$\F=\prod_{i=1}^r \mathfrak{p}_i^{n_i}$.
Assume that $K_\F\neq K_{\F\P_i^{-n_i}}$ and $|\mathbf{G}_i|>2$ for every $i$.
\begin{itemize}
\item[\textup{(i)}]
Let $s$ be an integer prime to $N$ such that the order of $[s]$ in $(\mathbb{Z}/N\mathbb{Z})^\times/\{\pm 1\}$ is an odd prime $p$ \textup{(}if any\textup{)}.
If $p>3$ or $|\mathbf{G}_i|>3$ for every $i$, then the special value
\begin{equation*}
\frac{g_{\left[\begin{smallmatrix}0\\ s/N\end{smallmatrix}\right]}(\theta)^m}{g_{\left[\begin{smallmatrix}0\\ 1/N\end{smallmatrix}\right]}(\theta)^m}
\end{equation*}
generates $K_{(N)}$ over $K$ as a real algebraic integer, where $m$ is an integer dividing $N$ for which $m(s^2-1)\equiv 0\pmod{\mathrm{gcd}(2,N)\cdot N}$ and $m\equiv N\pmod{2}$.

\item[\textup{(ii)}]
When $N$ is even, we set $C'=\big[(({N}/{2})\theta+t)\big]\in\mathrm{Cl}(\F)$ with $t\in\mathbb{Z}$ such that $t^2\equiv 1\pmod{N}$.
We further assume that 
\begin{equation*}
\left(\frac{g_\F(C')}{g_\F(C_0)}\right)^{\sigma(C)}\neq \frac{g_\F(C')}{g_\F(C_0)}
\end{equation*}
for every $C\in \mathbf{G}\setminus\{1\}$.
If $4$ divides $Nn$, then the special value
\begin{equation*}
e^{ \frac{ 2t }{N}\pi i} \frac{g_{\left[\begin{smallmatrix}1/2\\ t/N\end{smallmatrix}\right]}(\theta)^4}{g_{\left[\begin{smallmatrix}0\\ 1/N\end{smallmatrix}\right]}(\theta)^4}
\end{equation*}
is a real algebraic integer and generates $K_{(N)}$ over $K$.
In particular, if $4$ divides $N$ then the special value
\begin{equation*}
e^{ \frac{ t }{N}\pi i} \frac{g_{\left[\begin{smallmatrix}1/2\\ t/N\end{smallmatrix}\right]}(\theta)^2}{g_{\left[\begin{smallmatrix}0\\ 1/N\end{smallmatrix}\right]}(\theta)^2}
\end{equation*}
also generates $K_{(N)}$ over $K$ as a real algebraic integer.

\end{itemize}

\end{theorem}
\begin{proof}
\begin{itemize}
\item[\textup{(i)}]
Let $C'=\big[(s)\big]\in\mathrm{Cl}(\F)$.
Then the order of $C'$ in $\mathrm{Cl}(\F)$ is $p$, and so we see from Theorem \ref{main theorem2} that
\begin{equation*}
K_{(N)}=K\left(\frac{g_\F(C')}{g_\F(C_0)}\right).
\end{equation*}
And, the function 
\begin{equation*}
\frac{g_{\left[\begin{smallmatrix}0\\s/N\end{smallmatrix}\right]}(\tau)^{m}}{g_{\left[\begin{smallmatrix}0\\ 1/N\end{smallmatrix}\right]}(\tau)^{m}}
\end{equation*}
lies in $\mathcal{F}_N$, from which we attain
\begin{equation*}
\gamma=\frac{g_{\left[\begin{smallmatrix}0\\s/N\end{smallmatrix}\right]}(\theta)^{m}}{g_{\left[\begin{smallmatrix}0\\ 1/N\end{smallmatrix}\right]}(\theta)^{m}}\in K_{(N)}
\end{equation*}
by Proposition \ref{complex multiplication} and \ref{quad relation}.
Since $\gamma^{12N/m}={g_\F(C')}/{g_\F(C_0)}$, the special value $\gamma$ also generates $K_{(N)}$ over $K$ as a real algebraic integer by Lemma \ref{real generator}.
Note that $m=N$ always satisfies the condition $m(s^2-1)\equiv 0\pmod{\mathrm{gcd}(2,N)\cdot N}$.
Indeed, if $N$ is even, then $s^2-1$ must be even because $s$ is prime to $N$.

\item[\textup{(ii)}]
Now that 
\begin{equation*}
\left(\frac{N}{2}\theta+t\right)^2\equiv -\frac{Nn}{4}N+ tN\theta+t^2\equiv 1 \pmod{N\O_K},
\end{equation*}
the ray class $C'$ is of order $2$ in $\mathrm{Cl}(\F)$.
Then it follows from Remark \ref{additional case} that
\begin{equation*}
K_{(N)}=K\left(\frac{g_\F(C')}{g_\F(C_0)}\right).
\end{equation*}
And, the function 
\begin{equation*}
e^{ \frac{ 2t }{N}\pi i}\frac{g_{\left[\begin{smallmatrix}1/2\\t/N\end{smallmatrix}\right]}(\tau)^{4}}{g_{\left[\begin{smallmatrix}0\\ 1/N\end{smallmatrix}\right]}(\tau)^{4}}
\end{equation*}
lies in $\mathcal{F}_N$, which yields
\begin{equation*}
\gamma=e^{ \frac{2t}{N} \pi i}\frac{g_{\left[\begin{smallmatrix}1/2\\t/N\end{smallmatrix}\right]}(\theta)^{4}}{g_{\left[\begin{smallmatrix}0\\ 1/N\end{smallmatrix}\right]}(\theta)^{4}}\in K_{(N)}
\end{equation*}
by Proposition \ref{complex multiplication} and \ref{quad relation}.
Since $\gamma^{3N}={g_\F(C')}/{g_\F(C_0)}$, the special value $\gamma$ also generates $K_{(N)}$ over $K$ and is a real algebraic integer by Lemma \ref{real generator}.
\par
In particular, if $4$ divides $N$ then we have
\begin{equation*}
e^{ \frac{ t }{N}\pi i}\frac{g_{\left[\begin{smallmatrix}1/2\\t/N\end{smallmatrix}\right]}(\tau)^{2}}{g_{\left[\begin{smallmatrix}0\\ 1/N\end{smallmatrix}\right]}(\tau)^{2}}\in\mathcal{F}_N
\end{equation*}
by Proposition \ref{quad relation}.
And, in like manner, we get the conclusion.
\end{itemize}
\end{proof}

\begin{remark}
Recently, Jung et al. \cite{J-K-S} proved that the if $N\equiv d_K\equiv 0\pmod{4}$ and $|d_K|\geq 4N^{\frac{4}{3}}$, then the ray class $C'=\left[(({N}/{2})\theta+({N}/{2})+1)\right]\in\mathrm{Cl}(\F)$ satisfies
\begin{equation*}
\left(\frac{g_\F(C')}{g_\F(C_0)}\right)^{\sigma(C)}\neq \frac{g_\F(C')}{g_\F(C_0)}
\end{equation*}
for every $C\in \mathbf{G}\setminus\{1\}$.
Thus if $K_\F\neq K_{\F\P_i^{-n_i}}$ and $|\mathbf{G}_i|>2$ for every $i$, the special value
\begin{equation*}
e^{ (\frac{1 }{2}+\frac{1 }{N})\pi i} \frac{g_{\left[\begin{smallmatrix}1/2\\ 1/2+ 1/N\end{smallmatrix}\right]}(\theta)^2}{g_{\left[\begin{smallmatrix}0\\ 1/N\end{smallmatrix}\right]}(\theta)^2}
\end{equation*}
generates $K_{(N)}$ over $K$ as a real algebraic integer by Theorem \ref{real ray class invariant}.
\end{remark}

\begin{corollary}\label{real ray class invariant2}
With the notations and assumptions as in Theorem \ref{real ray class invariant}, 
let $p$ be an odd prime satisfying $p^2\,|\,N$ \textup{(}if any\textup{)}.
If $p>3$ or $|\mathbf{G}_i|>3$ for every $i$, then the special value
\begin{equation*}
\frac{g_{\left[\begin{smallmatrix}0\\ 1/p+1/N\end{smallmatrix}\right]}(\theta)^m}{g_{\left[\begin{smallmatrix}0\\ 1/N\end{smallmatrix}\right]}(\theta)^m}
\end{equation*}
generates $K_{(N)}$ over $K$ as a real algebraic integer, where
\begin{equation*}
m=\left\{
\begin{array}{ll}
p&\textrm{if $N$ is odd}\\
2p&\textrm{if $N$ is even}.
\end{array}\right.
\end{equation*}

\end{corollary}
\begin{proof}
Let $s=1+N/p$.
For a positive integer $i$, we have
\begin{equation*}
s^i\equiv 1+\frac{N}{p}i\pmod{N}
\end{equation*}
and hence the ray class $\big[(s)\big]\in\mathrm{Cl}(\F)$ is of order $p$ and $m(s^2-1)\equiv 0\pmod{\mathrm{gcd}(2,N)\cdot N}$.
Therefore, the corollary follows from Theorem \ref{real ray class invariant} \textup{(i)}.
\end{proof}

Now, we are ready to apply the ray class invariants in Theorem \ref{real ray class invariant} to the quadratic Diophantine equations described in Proposition \ref{Diophantine}.

\begin{example}\label{last example}
Let $K=\mathbb{Q}(\sqrt{-1})$, $\theta=\sqrt{-1}$ and $\F=9\O_K$.
Then $3$ is inert in $K$ and hence $|\mathbf{G}_i|>3$ for every $i$.
Since the ray class $C'=\big[(4)\big]$ in $\mathrm{Cl}(\F)$ is of order $3$, the special value
\begin{equation*}
\gamma=\frac{g_{\left[\begin{smallmatrix}0\\4/9\end{smallmatrix}\right]}(\theta)^{3}}{g_{\left[\begin{smallmatrix}0\\ 1/9\end{smallmatrix}\right]}(\theta)^{3}}
\end{equation*}
generates $K_{(9)}$ over $K$ as a real algebraic integer by Corollary \ref{real ray class invariant2}.
Here we observe that $\mathrm{Gal}(K_{(9)}/K)\cong W_{9,\theta}/\ker(\varphi_{9,\theta})$ by Proposition \ref{shimura reciprocity law}, and we have
\begin{equation*}
\begin{array}{rl}
W_{9,\theta}/\ker(\varphi_{9,\theta})=&\bigg\{ 
\footnotesize\left[\begin{matrix}
1&0\\
0&1
\end{matrix}\right],
\left[
\begin{matrix}
0&7\\
2&0
\end{matrix}\right],
\left[
\begin{matrix}
0&5\\
4&0
\end{matrix}\right],
\left[
\begin{matrix}
1&8\\
1&1
\end{matrix}\right],
\left[
\begin{matrix}
1&7\\
2&1
\end{matrix}\right],
\left[
\begin{matrix}
1&6\\
3&1
\end{matrix}\right],
\left[
\begin{matrix}
1&5\\
4&1
\end{matrix}\right],
\left[
\begin{matrix}
1&4\\
5&1
\end{matrix}\right],\vspace{0.2cm}\\
&
\footnotesize\left[\begin{matrix}
1&3\\
6&1
\end{matrix}\right],
\left[
\begin{matrix}
1&2\\
7&1
\end{matrix}\right],
\left[
\begin{matrix}
2&7\\
2&2
\end{matrix}\right],
\left[
\begin{matrix}
2&6\\
3&2
\end{matrix}\right],
\left[
\begin{matrix}
2&5\\
4&2
\end{matrix}\right],
\left[
\begin{matrix}
2&4\\
5&2
\end{matrix}\right],
\left[
\begin{matrix}
2&3\\
6&2
\end{matrix}\right],
\left[
\begin{matrix}
3&5\\
4&3
\end{matrix}\right],\vspace{0.2cm}\\
&
\footnotesize\left[\begin{matrix}
3&4\\
5&3
\end{matrix}\right],
\left[
\begin{matrix}
4&5\\
4&4
\end{matrix}\right]\bigg\},
\end{array}
\end{equation*}
where $\varphi_{9,\theta}$ is the homomorphism stated in Proposition \ref{ray class conjugate} .
Hence we obtain the minimal polynomial $f_9(X)$ of $\gamma$ over $K$ as
\begin{equation*}
{\footnotesize
\begin{array}{ccl}
f_9(X)&=&X^{18} -36X^{17}+234 X^{16}+1086 X^{15} + 2547 X^{14} + 12294 X^{13} +32415 X^{12}+41976 X^{11}\vspace{0.1cm}\\
&&+ 45459 X^{10}+55748 X^{9}+51480 X^{8}+ 22914 X^{7} -1092X^{6}-5310X^{5}-1719 X^{4}+6X^3\vspace{0.1cm}\\
&&+99X^2+18X+1
\end{array}}
\end{equation*}
and so we achieve $\mathrm{disc}\big(f_9(X)\big)=2^{54}\cdot 3^{135}\cdot 127^6\cdot 827^2$.
On the other hand, an odd prime $p$ satisfies $(\frac{-1}{p})=1$ if and only if $p\equiv1\pmod{4}$.
Therefore, if  $p\neq 2,3,127,827$, we see by Proposition \ref{Diophantine} that a prime $p$ can be expressed as $p=x^2+y^2$ for some $x,y\in\mathbb{Z}$ with conditions $x\equiv 1\pmod{9}$, $y\equiv 0\pmod{9}$ if and only if $p\equiv1\pmod{4}$ and $f_9(X)\equiv 0\pmod{p}$ has an integer solution.

\end{example}

\begin{example}\label{last example2}
Let $K=\mathbb{Q}(\sqrt{-5})$, $\theta=\sqrt{-5}$ and $\F=4\O_K$. 
Then $2$ is ramified in $K$ and so $|\mathbf{G}_i|>2$ for every $i$.
In a similar way as in Example \ref{even example} one can show that the ray class $C'=\big[(2\sqrt{-5}+1)\big]\in\mathrm{Cl}(\F)$ satisfies
\begin{equation*}
\left(\frac{g_\F(C')}{g_\F(C_0)}\right)^{\sigma(C)}\neq \frac{g_\F(C')}{g_\F(C_0)}
\end{equation*}
for every $C\in \mathbf{G}\setminus\{1\}$.
Thus the special value
\begin{equation*}
\gamma=e^{ \frac{ \pi i }{4}}\frac{g_{\left[\begin{smallmatrix}1/2\\1/4\end{smallmatrix}\right]}(\theta)^{2}}{g_{\left[\begin{smallmatrix}0\\ 1/4\end{smallmatrix}\right]}(\theta)^{2}}
\end{equation*}
generates $K_{(4)}$ over $K$ as a real algebraic integer by Theorem \ref{real ray class invariant}.
And, we get the minimal polynomial $f_4(X)$ of $\gamma$ over $K$ as follows:
\begin{equation*}
f_4(X)=X^{8}+16X^{7}-12X^{6}+16X^{5}+38X^{4}-16X^{3}-12X^{2}-16X+1.
\end{equation*}
On the other hand, the discriminant of $f_4(X)$ is $2^{68}\cdot 5^4$ and we derive that an odd prime $p$ satisfies $(\frac{-5}{p})=1$ if and only if
$p\equiv1,3,7,9\pmod{20}$.
Therefore, if $p\neq 2,5$ we conclude that a prime $p$ can be written in the form $p=x^2+5y^2$ for some $x,y\in\mathbb{Z}$ with conditions $x\equiv 1 ~(\bmod{~4}),~ y\equiv 0~(\bmod{~4})$ if and only if $p\equiv1,3,7,9\pmod{20}$ and $X^{8}+16X^{7}-12X^{6}+16X^{5}+38X^{4}-16X^{3}-12X^{2}-16X+1 \equiv0\pmod{p}$ has an integer solution.

\end{example}

\bibliographystyle{amsplain}

\address{
Ja Kyung Koo\\
Department of Mathematical Sciences \\
KAIST \\
Daejeon 305-701 \\
Republic of Korea} {jkkoo@math.kaist.ac.kr}
\address{
Dong Sung Yoon\\
National Institute for Mathematical Sciences \\
Daejeon 305-811 \\
Republic of Korea} {dsyoon@nims.re.kr}

\end{document}